\documentclass[12pt]{amsart}
\usepackage[margin=1in]{geometry}
\usepackage{color}
\usepackage{xcolor}
\usepackage{url}
\usepackage{enumerate}

\usepackage{amsmath}
\usepackage{amssymb}
\usepackage{cite} 
\usepackage{bm}
\usepackage{amsthm}
\usepackage{hyperref}
\usepackage{mathrsfs}
\allowdisplaybreaks

\newtheorem{remark}{Remark}
\newtheorem{theorem}{Theorem}
\newtheorem{proposition}{Proposition}
\newtheorem{lemma}{Lemma}
\newtheorem{corollary}{Corollary}

\newcommand{\e}{\boldsymbol{e}_3}

\newcommand{\rmm}{\boldsymbol{r}}
\newcommand{\rmmt}{\rmm'}

\newcommand{\umm}{\boldsymbol{u}}
\newcommand{\vmm}{\boldsymbol{v}}
\newcommand{\wmm}{\boldsymbol{w}}

\newcommand{\Umm}{\boldsymbol{U}}
\newcommand{\Vmm}{\boldsymbol{V}}
\newcommand{\Wmm}{\boldsymbol{W}}
\newcommand{\Ummt}{\wt\Umm}

\newcommand{\ummt}{\wt\umm}
\newcommand{\wmmt}{\wmm_{\boldsymbol r}}

\newcommand{\ck}{\check k}
\newcommand{\cm}{\check m}
\newcommand{\cn}{\check n}

\newcommand{\tu}{\wt{\boldsymbol{u}}}

\newcommand{\leray}{\mathcal{P}}
\newcommand{\lpv}{\boldsymbol {L}_{pv}}
\newcommand{\C}{\mathbb C}

\newcommand{\V}{\boldsymbol V}
\newcommand{\Le}{\mathcal L}
\newcommand{\N}{\mathbb N}
\newcommand{\R}{\mathbb R}
\newcommand{\Z}{\mathbb Z}
\newcommand{\T}{\mathbb T^3}

\newcommand{\dive}{\nabla \cdot}

\newcommand{\Ai}{\mathcal A_i}
\newcommand{\Aii}{\mathcal A_{i+1}}

\newcommand{\rfp}{r^{+}_k}
\newcommand{\rfm}{r^{-}_k}

\newcommand{\rfmm}{r^{-}_m}

\newcommand{\cip}[2]{\left( #1 \cdot #2\right)}

\newcommand{\skm}{\sum_{k,m,n;conv}}

\newcommand{\ip}[1]{\left\langle #1 \right\rangle}

\newcommand{\ds}[1]{\left(-\Delta\right)^{#1}}
\newcommand{\dsn}[1]{\left(-\Delta_\eta\right)^{#1}}

\newcommand{\bs}{\boldsymbol}

\def\mq{{\mathfrak q}}

\def\ind{{\bs 1}}
\def\mq{{\mathsf q}}
\def\ri{{\mathrm i}}
\def\mQ{{\mathsf{Q}_{EL}}}
\def\wt{\widetilde}
\def\pD{{\mathsf D}}
\def\rd{{\mathrm d}}
\def\ri{{\mathrm i}}
\def\vn{{V_{n, \sigma_1, \sigma_2}}}
\def\bpm{\begin{pmatrix}}
\def\epm{\end{pmatrix}}
\newcommand{\cN}{{\mathcal N}}

\def\ri{{\mathrm i}}
\def\sL{{\mathsf L}}
\makeatletter
\@namedef{subjclassname@2020}{\textup{2020} Mathematics Subject Classification}
\makeatother

\begin{document}
\title[near resonant approximation of rotating stratified Boussinesq system]{near resonant approximation of  the rotating stratified Boussinesq system  on a 3-torus}
\author{Bin Cheng}
\address{Department of Mathematics, University of
          Surrey, Guildford, GU2 7XH, United Kingdom}
   \email{b.cheng@surrey.ac.uk}
   
\author{Zisis N. Sakellaris}
\address{Department of Mathematics, University of
          Surrey, Guildford, GU2 7XH, United Kingdom}
   \email{z.sakellaris@surrey.ac.uk}
\date{}
\maketitle
\begin{abstract} Based on a novel treatment of near resonances, we introduce a new  approximation for the  rotating stratified Boussinesq system on  three-dimensional tori with  arbitrary aspect ratios.  The rotation and stratification parameters are arbitrary and not equal. We obtain global existence for the proposed  nonlinear system for  arbitrarily large initial data. This  system is sufficiently accurate, with an important feature  of coupling effects between slow and fast modes.  The  key to global existence is a sharp counting of the relevant number of nonlinear interactions. An additional regularity advantage arises from a careful examination of some mixed type interaction coefficients. In a wider context,  the significance of our near resonant approach  is a delicate balance between the inclusion of more interacting modes and the improvement of regularity properties, compared to the well-studied singular limit approach   based on exact resonance. 
\end{abstract}

\begin{section}{Introduction}
Let $\T:=[0, 2 \pi \sL_1] \times [0, 2 \pi \sL_2] \times [0,2\pi]$ be the three-dimensional flat torus with anisotropic  periods $2\pi\sL_1,2\pi\sL_2 >0$. We are interested in approximating solutions $(\Umm^\backprime,\rho)$, with $\Umm^\backprime:\T\times \R^+\to \R^3$,  $\rho:\T \times \R^+ \to \R$, to the unforced, rotating stratified Boussinesq system 
\begin{equation}\label{bous0}
\begin{cases}
&\partial_t \Umm^\backprime + \Umm^\backprime\cdot \nabla\Umm^\backprime- \nu_1 \Delta \Umm^\backprime + \Omega  \e \times \Umm^\backprime - N \rho \e = -\nabla p\\
& \partial_t \rho +\Umm^\backprime\cdot\nabla\rho - \nu_2 \Delta \rho + N(\Umm^\backprime\cdot \e)= 0\\
& \dive \Umm^\backprime=0,
\end{cases}
\end{equation}
with zero-mean initial data $(\Umm^\backprime_0,\rho_0)\in H^{\ell} (\T;\R^4)$,  $\ell\ge 1$ and  $\nabla\cdot \Umm^\backprime_0=0$.  Here $\e=(0,0,1)^\intercal$ { and $H^{\ell} $ is the Sobolev space of  order $\ell$ on  $\T$}.

System \eqref{bous0}  is a well-known model of geophysical fluid dynamics (GFD). It describes the  nonlinear dynamics of an incompressible fluid with  velocity $\Umm^\backprime$ and density deviation $\rho$ from a linear background state,  with $p$ standing for the pressure.  In this context, the influence of  rotation is measured  via $\Omega>0$ and that of stratification via the Brunt-V\"ais\"al\"a frequency $N>0$. The positive parameters $\nu_1, \nu_2$ represent the viscosity and heat conductivity, respectively, { following the mathematical GFD conventions, see e.g. \cite[Chapter 1]{MAJ} for more details.}  Note that we do not require largeness or smallness conditions on these parameters. The relative strength of the effects of rotation and stratification is measured via $\eta= \frac \Omega N$  satisfying $\eta\ne1$.  In the case  $N=\rho=0$, the Boussinesq system  reduces to the incompressible rotating Navier-Stokes equations. 

The pressure term can be removed from \eqref{bous0} using the  projection $\leray = \begin{pmatrix}
\leray^\backprime&0\\
0&1
\end{pmatrix},$
where $\leray^\backprime$ is the ordinary Leray projection to divergence-free fields in three dimensions. Then, using a  four-component field $\Umm := ( \Umm^\backprime, \rho)$, 
 we obtain the system
\begin{equation}\label{boussi0}
\partial_t \Umm + \leray (\Umm^\backprime\cdot \nabla \Umm) -  \boldsymbol{\nu}  \leray    \Delta \leray  \Umm = N \Le \Umm \quad\text{with}\quad \nabla\cdot \Umm=0\quad\text{and}\quad \Umm_0=(\Umm^\backprime_0,\rho_0),
\end{equation}
where 
\begin{equation*}\Le =\leray \begin{pmatrix}
\eta J & 0\\
0 & J
\end{pmatrix}\leray,\quad   \boldsymbol{\nu} =\begin{pmatrix}
\nu_1  1_{3\times 3} & 0\\
0 & \nu_2 
\end{pmatrix},\quad  J = \begin{pmatrix}
0&1\\
-1&0
\end{pmatrix},
\end{equation*}
and $\Umm_0\in  H^{\ell} (\T;\R^4)$, $\ell\ge 1$, is divergence-free.  Here and below, a 4-vector field is said to be divergence-free if its first 3 components,  i.e. the velocity field, are divergence-free.  As the linear operator  $N\Le$ in the PDE system is  responsible for inertia-gravity waves,  we can use the associated linear  evolution operator $e^{N t \Le}$    to   transform our system  via setting
  $\umm =e^{- N t \Le} \Umm$ and $B(\Umm,\Umm) =  \leray(\Umm^\backprime\cdot\nabla\Umm)$.  Then an equivalent system  to \eqref{boussi0} can be obtained:
\begin{equation}\label{boussmod0}
\partial_t \umm + B(Nt,\umm, \umm) + A \umm =0 \quad\text{with}\quad \nabla\cdot \umm=0 \quad\text{and}\quad \umm_0=\Umm_0,
\end{equation}
where 
\[ B(Nt,\umm,\umm) = e^{- N t \Le} B(e^{N t \Le}\umm,e^{N t \Le}\umm)\quad\text{and}\quad A  = -e^{- N t \Le} \leray\boldsymbol{\nu} \Delta e^{N t \Le}  \umm.\] In this so-called modulated form,  nonlinear effects are expressed through the transformed bilinearity $ B(Nt,\cdot,\cdot)$. This filtering technique has been widely used as a method to tackle  problems for oscillatory perturbations to evolution equations,  see e.g. \cite{SCHO}.

 Since the near resonance  concept that we study can be  regarded as  a finite version of the well  known exact resonance, and since both notions are based on nonlinear interactions, we begin with explicating the nonlinear effects in systems \eqref{boussi0} and  \eqref{boussmod0}  using Fourier series.  For a zero-mean field $\Umm\in L^2(\T;\C^4)$, we  expand $\Umm(x)= \sum_{k\in(\Z^3\setminus\{\vec 0\})}  e^{\ri \ck\cdot x} U_k$  where
\[
\ck=(k_1/\sL_1, k_2/\sL_2,k_3)^\intercal
\]
and  $U_k$  are the Fourier coefficients of $\Umm$. As illustrated in Section \ref{preliminariesS}, the symbol of $\Le $ is given by a $4 \times 4$ anti-Hermitian matrix.  Its spectrum consists of 0  with multiplicity two, and 
\[
\pm \ri\omega_k = \pm  \ri\frac{\sqrt{\ck_1^2 + \ck_2^2 + \eta^2 k_3^2}} {|\ck|}\quad\text{for}\quad k=(k_1,k_2,k_3)^\intercal \in \Z^3\setminus \{\vec 0 \},
\]
see e.g. \cite{EM1}. Moreover, there exist corresponding eigenvectors  $\{r_k^{00},  r_k^{0}, r_k^+,r_k^-\}$ that form an orthonormal basis of $\C^4$.  Note that
all our study does not involve  the subspace  spanned by $e^{\ri \ck\cdot x}r_k^{00}$, which   on the physical side consists of any 4-vector field of which the velocity component is a  potential flow and the density component vanishes. Then  define projections
 \[\leray_k^{\sigma}\Umm(x) =  e^{\ri \ck\cdot x} \cip {U_k}{\overline {r_k^{\sigma}}}r_k^{\sigma}\quad\text{for  }\sigma \in \{0,\pm 1 \}:= \{0, \pm \},\] that mutually cancel each other, and express the bilinearity in \eqref{boussi0} as a convolution sum
\begin{equation}\label{convolutionsign}
B(\Umm,\Vmm) = \sum_{\substack{k,m,n\in\Z^3\setminus\{\vec 0\}\\ k+m+n=\vec 0}}\sum_{\sigma_1,\sigma_2,\sigma_3 \in \{0,\pm \}} \leray_{-n}^{-\sigma_3} B(\leray_k^{\sigma_1} \Umm,\leray_m^{\sigma_2}\Vmm),
\end{equation}
where the arguments $\Umm,\Vmm$ are divergence-free with zero mean.

 Quantities related to $r_k^0$ are customarily labelled slow, whereas those  related to $r_k^{\pm}$ fast.  Then, for any  weakly divergence-free vector field  $\Umm\in L^2 (\T;\C^4)$, we can define slow and fast projections as follows:
\[
\Umm_s = \sum_{k\in\Z^3} \leray_k^0 \Umm\quad\text{and}\quad \Umm_f =\Umm - \Umm_s.\]
The slow part $\Umm_s$ is in the kernel of $\Le$, whereas the fast part $\Umm_f$ is in the orthogonal complement of the kernel of $\Le$ with the orthogonality in terms of the $L^2(\T;\C^4)$ inner product.

The eigenexpansion formalism can also be used for the modulated bilinearity of \eqref{boussmod0}.
For divergence-free and zero-mean vector fields $\Umm, \Vmm$, the expansion is:
\begin{equation}\label{mod0}
B(\tau,\Umm,\Vmm)=\sum_{\substack{k,m,n\in\Z^3\setminus\{\vec 0\}\\ k+m+n=\vec 0}}\sum_{\sigma_1,\sigma_2,\sigma_3 \in \{0,\pm \}} e^{\ri \omega_{kmn}^{\vec \sigma}\tau}\leray_{-n}^{-\sigma_3} B(\leray_k^{\sigma_1} \Umm,\leray_m^{\sigma_2}\Vmm),
\end{equation}
where  $\tau=Nt$ and
\begin{equation}\label{omegadef}
\quad\omega_{kmn}^{\vec \sigma}:=\sigma_1 \omega_k +\sigma_2 \omega_m + \sigma_3  \omega_n, \quad\text{for}\quad\vec \sigma= (\sigma_1,\sigma_2,\sigma_3)\in\{0,\pm\}^{3}.
\end{equation}
Under our notational convention, the input to the bilinearity is represented via the modes with wavevectors $k, m$, and the output wavevector $-n$.

Exact resonances correspond to triplets $(k,m,n)\in(\Z^3\setminus\{\vec 0\})^3$ satisfying  $\omega_{kmn}^{\vec \sigma}=0$ and $k+m+n=\vec 0$. Restrictions of nonlinear interactions  to  resonant modes play a decisive role in systems that arise from \eqref{boussi0}, in the limits $N,\Omega\to\infty$ with fixed $\eta$,  \cite[Lemma 4.1]{BMN1}. The slow-fast  dichotomy,  reflected in the sign of each of the three interacting modes,  results in a decoupling in the aforementioned resonant limits.  In the case of purely slow interactions, the 3D quasigeostrophic (3D-QG) system is obtained  via restricting the bilinearity in  \eqref{mod0}  to only  $\vec \sigma=(0,0,0)$, since no other types of interaction sets with slow output modes can occur in the exact resonance setting \cite{EM1}. In particular, 3D-QG has been shown to approximate the dynamics of \eqref{boussi0} in various limiting settings \cite{BB1}, \cite{BMN1}, \cite{BMN2}, \cite{BMNZ}, \cite{EM1}, \cite{EM2}, \cite{GAL1}, \cite{GAL2}.  Nevertheless, the fast ageostrophic (AG) part  of the solution, containing but not limited to fast-fast-fast (FFF) wave interactions, plays a significant role as well in the exact resonance setting \cite{BMN1}.

Our main goal is to accurately approximate \eqref{boussi0}, under no limiting considerations, with a well-posed system.  For that, we approximate the original bilinearity using  a near resonant bilinearity $\wt B(Nt,\cdot,\cdot)$, which  takes into account fundamentally more input and output interactions between the $(k, m, n)$ triplet modes  compared to those of an exact resonance  approximation.  We also emphasize that the global solvability of our system is independent of physical parameters $N,\Omega$, in contrast to  aforementioned results.
\begin{subsection}{Near resonant approximation based on restricted interactions}
The eigenmode  expansion of $B(N t, \cdot,\cdot)$ in \eqref{mod0} allows us to rigorously define a restricted bilinearity $\wt B(Nt,\cdot,\cdot)$, based on a relaxation of the exact resonance notion.  In more detail, the proposed restriction on the level of interaction sets  is quantified by two so-called bandwidths that depend on  $\max\{|\ck|,|\cm|,|\cn|\}$, and satisfy
\begin{equation}\label{deltabound}
\delta(k,m,n),\delta^* (k,m,n)\in \left [0,\min\{\tfrac {\eta} 2, \tfrac 1 2\}\right).
\end{equation}
The purpose of $\delta$ and $\delta^*$ is to control the magnitude $|\omega_{kmn}^{\vec \sigma}|$ in two different settings. 
 On one hand, we define the   restricted  fast-fast-fast (FFF) interactions  set
\begin{equation}\label{nrrnsfast}
\cN^{\textnormal {FFF}}:= \{(k,m,n)\in (\Z^3 \setminus \{\vec 0\})^3: k+m+n=\vec 0, \min_{\vec \sigma \in\{\pm\}^3} | \omega_{kmn}^{\vec \sigma} |\leq\delta(k,m,n)\},
\end{equation}
with
\[
\delta(k,m,n) \lesssim \min \{|\ck|^{-1},|\cm|^{-1},|\cn|^{-1}\}.
\]

On the other hand, we define  mixed near resonant subsets, for fast-fast-slow (FFS) and fast-slow-fast (FSF) interactions, as follows
\begin{align}
& \cN^{\textnormal {FFS}} := \{(k,m,n) \in (\Z^3\setminus\{\vec0\})^3 : k+m+n=\vec 0,  |\omega_k - \omega_m|\leq \delta^*(k,m,n)\}{,}\\
& \cN^{\textnormal {FSF}} := \{(k,m,n) \in (\Z^3\setminus\{\vec0\})^3 : k+m+n=\vec 0,  |\omega_k - \omega_n|\leq \delta^*(k,m,n)\}{,}
\end{align}
under the assumption that
\[
\delta^* (k,m,n)\lesssim \min \{|\ck|^{-\xi},|\cm|^{-\xi},|\cn|^{-\xi}\},\quad\text{for a  parameter}\quad\xi>0.
\]

Under our notational convention, the first two superscripts  of the mixed  sets stand for the signs of the input modes, with the last one reserved for that of the output.
By \eqref{deltabound} and  $\min\{\eta, 1\}\leq\omega_k$ for all  $k\in \Z^3,$ there are no near resonant interactions between one fast and two slow modes, namely there are no FSS, SFS and SSF interactions. Also,  in view of \eqref{omegadef}, all slow-slow-slow (SSS) interactions are  trivially exact resonances. Finally, we will make no restrictions on slow-fast-fast (SFF) interactions  since this is permitted as far as global existence is concerned.

  In order to use our newly introduced sets to restrict the bilinearity of \eqref{boussmod0}, it is necessary to keep track of the interaction signs involved. Hence, we introduce the following  notation:
  \begin{equation}\label{bcoeff}
B_{kmn}^{\sigma_1\sigma_2\sigma_3} (\tau,\Umm,\Vmm):=  e^{\ri \omega_{kmn}^{\vec \sigma}\tau}\leray_{-n}^{-\sigma_3} B(\leray_k^{\sigma_1} \Umm,\leray_m^{\sigma_2}\Vmm),
\end{equation}
with the convention that $B_{kmn}^{\sigma_1\sigma_2\sigma_3} (\Umm,\Vmm) = B_{kmn}^{\sigma_1\sigma_2\sigma_3} (0,\Umm,\Vmm)$. Here and in what follows, for the sake of brevity, we write $k,m,n;conv$ in order to designate the summing index, $k,m,n\in \Z^3\setminus\{\vec 0\}$, $k+m+n=\vec 0$, of any convolution sum that occurs throughout the text. In addition, the default range for any sign occurring in the following sums is $\{0,\pm\}$, up to the specified constraints.
Then,  the slow output part of the original bilinearity  can be expanded as follows
\begin{align*}
B_s (\tau,\Umm, \Vmm)
&=\skm B_{kmn}^{0 0 0}(\Umm,\Vmm)  +\sum_{\substack{k,m,n;conv}}\sum_{(\sigma_1,\sigma_2)\neq (0,0)} B_{kmn}^{\sigma_1 \sigma_2 0}(\tau, \Umm,\Vmm).
\end{align*}
In an analogous manner, the fast output terms can be expanded into terms with purely fast and mixed input as follows
\begin{align*}
B_f (\tau,\Umm, \Vmm)& = \sum_{\substack{k,m,n;conv}}\sum_{\sigma_1\sigma_2 \sigma_3 \neq 0}  B_{kmn}^{\sigma_1 \sigma_2 \sigma_3}(\tau, \Umm,\Vmm) +\sum_{\substack{k,m,n;conv}}\sum_{\substack{\sigma_1  \sigma_2  = 0}} \sum_{\sigma_3 \neq 0} B_{kmn}^{\sigma_1 \sigma_2 \sigma_3}(\tau, \Umm,\Vmm).
\end{align*}

Equipped with this notation, we
  define the approximate terms with slow output as follows:
\begin{equation}\label{restrictedBS}
\wt B_s (\tau,\Umm, \Vmm)  = \skm   B_{kmn}^{000}(\Umm,\Vmm) +\sum_{\substack{k,m,n;conv}}\sum_{\sigma_1 \sigma_2 < 0}   B_{kmn}^{\sigma_1 \sigma_2 0}(\tau, \Umm,\Vmm)\ind_{\cN^{\textnormal {FFS}}}(k,m,n),
\end{equation}
where { $\ind_{A}(\cdot,\cdot,\cdot)$ is the characteristic  function of a set $A$}. 
In contrast to the original slow  bilinearity, the input fields are now  either  both fast or both slow since, again, the restriction on $\delta^*$ in \eqref{deltabound}  trivially excludes SFS and FSS terms in our approximation.  We emphasize  that the $\tau$-independent SSS part of the bilinearity is unrestricted in \eqref{restrictedBS} and must all be considered in any notion of resonances. Finally, we note that  FFS interactions with the same fast sign have been excluded from our considerations, as the corresponding values of $|\omega_{kmn}^{\vec\sigma}|$ are bounded from below by $2 \min \{ \eta, 1\}$  which is greater than the upper bound in \eqref{deltabound}.

 In a similar manner, we define an approximate fast bilinearity as follows
\begin{align}\label{restrictedBF}
&\wt B_f (\tau,\Umm,\Vmm): = \sum_{\substack{k,m,n;conv}} \sum_{\sigma_1\sigma_2\sigma_3 \neq 0} B_{kmn}^{\sigma_1 \sigma_2 \sigma_3}(\tau, \Umm,\Vmm) \ind_{\cN^{\textnormal {FFF}}}(k,m,n)\\
&+\skm \sum_{\substack{\sigma_1 \sigma_3 < 0}}   B_{kmn}^{\sigma_1 0 \sigma_3}(\tau, \Umm,\Vmm)\ind_{\cN^{\textnormal {FSF}}}(k,m,n)+\skm \sum_{\sigma_2  \sigma_3 \neq 0}   B_{kmn}^{0 \sigma_2  \sigma_3 }(\tau, \Umm,\Vmm). \nonumber
\end{align}
Our fast approximation only excludes  SSF interactions, with restrictions posed on the  FFF and FSF subsets. The latter is also subjected to a sign restriction, similarly to the slow mixed term. Nevertheless, the SFF part of the bilinearity is identical to the original.

Since $\nu_1 \neq \nu_2$ in general,  a further approximation is needed for the Laplacian terms. In particular, the  non-commutativity of the matrix $\boldsymbol{\nu}$ with fast and slow projections introduces unwanted mixing effects.  Thus, we consider  the modified diagonal  Laplacian operator
\begin{equation}\label{laplacianmod}
 \wt A \wt \Umm : =  -\wt\nu_{11}  \Delta \Ummt_s - \wt\nu_{22}  \Delta \Ummt_f.
\end{equation}
Here, $\wt\nu_{11}$ and $\wt\nu_{12}$ denote { the} scalar pseudodifferential operators of degree 0, depending linearly on $\nu_1$ and $\nu_2$, which result  from eliminating the fast scale in $A$. { Their definition based on this strategy is given in \eqref{visco1}-\eqref{visco2}.}
     In other words, $\wt A$ is independent of $\tau = Nt.$ Finally, we  note that $\wt A$ is elliptic with  ellipticity constant   $\nu_{min} := \min \{\nu_1, \nu_2\}$.

 We set
 \begin{equation}\label{restrictedB}
  \wt B(\Umm,\Vmm):= \wt B(0,\Umm, \Vmm),\quad\wt B_f(\Umm,\Vmm):= \wt B_f(0,\Umm, \Vmm),\quad\wt B_s(\Umm,\Vmm):= \wt B_s(0,\Umm, \Vmm)
\end{equation}
 for the restricted bilinearity of the original system  \eqref{boussi0}. In particular, 
\[
\wt B(\Ummt,\Ummt) = \wt B_s (\Ummt,\Ummt) + \wt B_f (\Ummt,\Ummt),
\]
 via \eqref{restrictedBS},\eqref{restrictedBF}, and \eqref{restrictedB} so that  we arrive at the near resonant approximation of the Boussinesq system
\begin{equation}\label{restrictedS}
\partial_t \Ummt  + \wt B(\Ummt,\Ummt) +  \wt A \Ummt = N \Le \Ummt.
\end{equation} 
In the absence of  the dissipative term, system \eqref{restrictedS} would  satisfy the fundamental property of $L^2$ energy conservation,  since the FFS and FSF sets crucially share a common bandwidth $\delta^*$. 
\end{subsection}
\begin{subsection}{Main results}
The existence part for our proposed system is covered by the following a-priori estimate.  Here and below, the dependence of a constant on other quantities is always in the form
of a positive-valued, smooth function. Dependence on $\T$ is just a shorthand notation for
dependence on $\sL_1 , \sL_2.$  All our estimates depend smoothly on $\sL_1,\sL_2$, and are independent of $N,\Omega$ unless stated otherwise.
\begin{theorem}\label{globalexistenceT}
Let $\nu_1, \nu_2,T>0$ and $\Ummt_0 \in H^1 (\T;\R^4)$ be a   divergence-free and zero-mean vector field. Consider a solution $\Ummt$ of the near resonant approximation to the Boussinesq system \eqref{restrictedS}
for $t\in [0,T],$ with initial data $\Ummt_0$. Moreover, let $\delta, \delta^*\in \left[0,\min\{\frac {\eta} 2, \frac 1 2\}\right)$ be the bandwidths for the sets $\cN^{\textnormal {FFF}}$ and $\cN^{\textnormal {FFS}}$, respectively.  If 
\[
\delta  \lesssim \min\{|\ck|^{-1},|\cm|^{-1},|\cn|^{-1}\}\quad\text{and}\quad\delta^*  \lesssim \min\{|\ck|^{-\frac 6 5},|\cm|^{- \frac 6 5},|\cn|^{- \frac 6 5}\},
\]
then there exist constants ${C_1= C_1}(\eta,\T)$ and ${C_2=C_2( \eta, \T, \nu_{min}^{-1}, \|\wt\Umm_0\|_{L^2}, { \|( \wt{\Umm}_0 )_s \|_{H^1}})}$ both of which also depend on the implied constants in the above conditions, so that
\begin{equation}\label{mains}
 \|\wt \Umm_s(T)\|_{H^1}^2+2\nu_{min}\int_0^T \|\Ummt_s\|_{H^2}^2 \rd t\leq C_1 \left( \nu_{min}^{-1} \|\wt \Umm_0\|_{L^2}^2 +\|( \wt{\Umm}_0 )_s\|_{H^1}\right)^{2}
\end{equation}
and
\begin{equation}\label{mainf2}
 \|\wt \Umm_f(T)\|_{H^1}^2 +\nu_{min}\int_0^T \|\Ummt_f \|_{H^2}^2\rd t  \leq { C_2} \|( \wt{\Umm}_0 )_f\|_{H^1}^2.
\end{equation}
\end{theorem}

Standard existence results on local in time  strong solutions
 \[
\Ummt\in C\left([0,T);H^{\ell}(\T;\R^4) \right) \cap L^2\left([0,T);H^{\ell+1}(\T;\R^4) \right)\quad \ell\geq \frac 1 2,
\] $T>0,$ for systems with Navier-Stokes type bilinearity \cite[Theorem 3.5]{CDGG} are applicable to our system. In particular,  the modified viscosity matrix $\wt{\boldsymbol{\nu}}$ is appropriately controlled due to Lemma \ref{ellipticity}, so that all the corresponding energy estimates go through. Crucially, our modified bilinearity obeys better product rules than an ordinary Navier-Stokes one in three dimensions, as we will prove in the sequel.   Thus, well-posedness of the approximate system follows from Theorem \ref{globalexistenceT}, by  arguments similar to  \cite[Theorem 3.5 and its remark, Corollary 3.1]{CDGG} and \cite[Lemma 6.5]{BCZNS}.\\
\begin{remark}
An interesting partial decoupling property can be observed in \eqref{mains}: the $L^{\infty}_t H^1_x$  and $L^{2}_t H^2_x$  norms of the slow part are bounded independently  of $\|( \wt{\Umm}_0 )_f\|_{H^1}$. This contrasts
with the coupling effect due to near resonance approximation of the bilinearities.
\end{remark}
The proof of Theorem \ref{globalexistenceT} is based on a separate treatment for the slow and fast parts of the solution, taking into account some good properties of the slow equation.  In more detail, we equivalently write the slow output part of \eqref{restrictedS} as a transport equation for the linear potential vorticity  $Q= \partial_1 U_2 - \partial_2 U_1 - \eta \partial_3 \rho$ which corresponds to the linear part of Ertel's potential vorticity. On the Fourier side,  the  linear potential vorticity is expressed as
\begin{equation}\label{eqlpv}
Q =   \ri \sum_{k\in \Z^3} e^{\ri \ck\cdot x}\sqrt{{ \ck_1^2+ \ck_2^2} + \eta^2 k_3^2}\cip{U_k}{r_k^0}.
\end{equation}
By introducing { an} operator $\lpv : H^1(\T;\R^4)\to L^2 (\T;\R)$ such that $\lpv \Umm =\lpv\Umm_s = Q$,  we  have an equivalent potential vorticity formulation of the slow equation as
\[
\partial_t Q + \lpv\wt B_s (\Umm,\Umm) - \wt \nu_{11} \Delta Q= 0.
\]
The slow bilinearity  $\lpv \wt B_s (\Umm,\Umm)$ splits into two parts: $\lpv B_s(\Umm_s, \Umm_s)$ and $\lpv \wt B_s(\Umm_f, \Umm_f)$,  while the slow-fast terms vanish by definition of $\wt B_s$. The purely slow part of the bilinearity has a favourable transport structure. On the other hand,  the slow mixed part, which is absent in the exact resonance setting,   benefits from interaction coefficients that are proportional to $\omega_{kmn}^{\vec\sigma}$ and counteract coupling, see Section \ref{slowP}. In fact, smaller values of  $\delta^*$ have a stronger  effect on the slow part. 

Further regularity gains, for both the slow and fast output parts of \eqref{restrictedS}, are established via lattice point counting methods. In particular, by a standard harmonic analysis argument, the number of interactions in a near resonance set determines the regularity properties in the estimates of the corresponding restricted convolution.  This  approach is standard by now in the exact resonance literature, see   \cite[Lemma 3.1]{BMN3} and \cite[Lemma 6.2]{CDGG}, but presents new challenges in the near resonance case.

The number of interactions in the sets $\cN^{\textnormal {FFF}},\cN^{\textnormal {FFS}}$ and $\cN^{\textnormal {FSF}}$  is estimated via two different approaches. For purely fast interactions included in $\cN^{\textnormal {FFF}}$,  it suffices to investigate the cardinality of the localised set
\[
\left\{k\in\Z^3\setminus\{\vec 0\}:\min_{\vec \sigma \in \{\pm\}^3}|\omega_{kmn}^{\vec \sigma}|\leq\delta,\; \frac 1 2 |\cn|\leq |\ck|\leq |\cn|\right\}\quad\text{for fixed}\quad {n\in\Z^3\setminus\{\vec 0\}},
\]
which turns out to be  $O(\delta |\cn|^3)$.  In detail,  the lattice counting problem in this case is converted to volume estimates for sublevel sets, up to a well-behaved error term.  We then   utilize  results proven in \cite{BCZNS} for the elliptic integrals arising from the volume integrals.  Finally, we  remark that our result differs by the corresponding result in \cite{BCZNS}, for the  rotating Navier-Stokes system,  by a logarithmic in $\delta$ factor.

 On the other hand, for mixed interactions included in $\cN^{\textnormal {FFS}}$ and $\cN^{\textnormal {FSF}}$, we essentially need to investigate  the cardinality of the following set:
\[
\left\{k \in Z^3\setminus\{\vec 0\}:  \left|f_{\eta}\left(\frac {k_3}{|\ck|}\right) - f_{\eta}\left(\frac {n_3} {|\cn|}\right) \right|\leq \delta^* ,\; M\leq |\ck| < 2M,\; |\ck|\leq |\cn| \right\}
\]
for fixed $M>0, n\in\Z^3\setminus\{\vec 0\}$,
where
 $$f_{\eta} (x) := \sqrt{1+ (\eta^2 -1)x^2 }.$$
 The cardinality turns out to be $O(\sqrt{\delta^*} M^{3}+M^{2})$, which is sharp as  shown in Appendix \ref{appendixlb}.

Theorem \ref{globalexistenceT} is the analog of the recently established result \cite[Theorem 1.1]{BCZNS} on the rotating Navier-Stokes system, with some crucial differences. In that work,  well-posedness  results were obtained for a near resonant  restricted  Navier-Stokes type system, with bandwidth  bounded as
$\delta_{NS} \log(\frac 1 {\delta_{NS}}) \lesssim \min \{c,|\ck|^{-1}, |\cm|^{-1}, |\cn|^{-1}\}$, for an absolute constant $c$.  
 Several subtleties arise in comparison. First,  for  the rotating Navier-Stokes system, slow modes are characterized  by a null vertical component, i.e. $k_3=0$, whereas in  \eqref{boussi0}
slow modes occur for every   wavevector, thus SSS resonance occurs for every triplet $k,m,{ n}$ in the convolution sum. Second,  whereas the corresponding rotating Navier-Stokes near resonance condition is:
\[
\left| \frac {k_3} {|\ck|} \pm \frac {m_3} {|\cm|}   \pm  \frac {n_3} {|\cn|} \right|  \leq \delta_{NS},
\]
the non-linearity of $f_\eta$ further complicates things. In fact, non-SSS near resonances  are characterized by:
\[
\left|\sigma_1 f_{\eta}\left(\frac {k_3} {|\ck|} \right) +\sigma_2 f_{\eta}\left(\frac {m_3} {|\cm|} \right) + \sigma_3 f_{\eta}\left(\frac {n_3} {|\cn|} \right)\right|  \leq \delta_\sigma, 
\]
where
\begin{equation}\label{mdelta}
 \delta_\sigma= \begin{cases}
 \delta,& \text{when}\quad \sigma_1\sigma_2\sigma_3 \neq 0\\
\delta^*, &\text{when}\quad (\sigma_1,\sigma_2,\sigma_3)=(\pm,\mp,0),(\pm,0,\mp)\\
2 \max \{\eta, 1\}, &\text{when}\quad \sigma_1=0,\sigma_2 \sigma_3 \neq 0\\
\min\{\eta, 1\} -\epsilon, &\text{for any remaining cases with }\vec \sigma \neq \vec 0
\end{cases}
\end{equation}
 with arbitrarily small $\epsilon>0$.
Note that, a strictly positive gap between $\delta_\sigma$ and 0 substitutes small-divisor estimates, in terms of control of   $|\omega_{kmn}^{\vec \sigma}|^{-1}$.
 Lastly, a-priori estimates for \eqref{restrictedS} have to take into account the slow-fast splitting of the system that will  be exploited in the proof of 
Theorem \ref{globalexistenceT}, in contrast to the unified approach of the rotating Navier-Stokes result.

We now proceed to the statement of our improved convolution sum estimates.  First, we present a result on the restricted FFF terms, in the spirit of  \cite[Theorem 4.1]{BMN1} for the exact resonance case.  We also mention  the corresponding  results on the rotating Navier-Stokes system, \cite[Theorem 3.1]{BMN3} and \cite[Theorem 1.3]{BCZNS}, in the exact and near resonance settings, respectively. The regularity for the restricted convolutions is measured in  Sobolev spaces, with the notation $\pD^{\ell} \Umm = \sum_{k\in\Z^3\setminus\{\vec 0\}} e^{\ri \ck \cdot x}|\ck|^{\ell}U_k $.
\begin{theorem}\label{restrictedfT} 
Let $\ell\in(0,1]$ and $\umm,\vmm \in H^{\ell+1}(\T;\R^4)$ be divergence-free and zero-mean vector fields. Moreover, in the FFF near resonance set \eqref{nrrnsfast}, let  $\delta \in \left[0, \min \{\frac {\eta} 2, \frac 1 2\}\right)$ satisfy
\[
\delta(k,m,n) \leq { C_\delta}\min \{|\ck|^{-1}, |\cm|^{-1}, |\cn|^{-1}\}.
\]
Then, the following estimate holds true  
\begin{equation}\label{restrictedfff1}
\left| \ip{\pD^{\ell} \wt B_f (\umm_f,\vmm_f),\pD^{\ell} \vmm_f} \right|  \lesssim  \|\umm_f \|_{H^1} \|  \vmm_f \|_{H^{\ell}}   \|  \vmm_f \|_{H^{\ell+1}},
\end{equation}
 with the implied constant depending on $\eta, \ell, \sL_1, \sL_2, { C_\delta}.$
\end{theorem}
Theorem \ref{restrictedfT} allows us to gain $\frac 1 2$ derivatives  when performing bilinear estimates in  three dimensions,  thus making the interactions $2D$-like. In more detail, given  two scalar functions $u,v$ defined on $\R^d$ or $\mathbb T^d$, and $a,b \in [0, \frac d 2),$ with $0<a+b$, we have
\[
\| u v\|_{H^{a+b - \frac d 2 } } \lesssim \| u \|_{H^{a  } }\|  v\|_{H^{b} },
\]
see Appendix \ref{appendixpe} and   \cite[Lemma 6.2]{CDGG}.
Thus, the derivative cost of estimating a product in homogeneous Sobolev spaces with positive index, which are not algebras, is given by $\frac d 2.$

Restrictions to the set $\cN^{\textnormal {FFS}}$ present more subtleties, due to the corresponding interaction coefficients being connected to the mixed interaction bandwidth, $\delta^*$. This phenomenon is not observed in the context of exact resonance, as the interaction coefficient is identically 0  when restricted to the exact resonance set --  see e.g. \cite{BMNZ}, \cite{EM1}.   Furthermore,  the importance of the interaction coefficients for the mixed terms in  near resonance    has not been extensively noted, to the best of our knowledge. An  exception is \cite{TW}, which focuses on long-time asymptotics for the Oceanic Primitive Equations and identifies the importance of interaction coefficients, as far as near resonant interactions are concerned. Moreover, in \cite{AOW}, among various results on numerics and multiscale asymptotic analysis for three fluid models, including \eqref{bous0}, the smallness of interaction coefficients  is highlighted,   in a different near resonant setting. Finally we mention \cite{OGW}, concerning the connection of interaction coefficients with  higher order wave interactions.

 The following Theorem provides convolution sum estimates for FFS terms.
\begin{theorem}\label{restrictedT}
Let $\ell\ge0$ and  $\umm,\wmm \in H^{\ell+1}(\T;\R^4)$ be divergence-free and  zero-mean vector fields.   
Consider any  $\xi \in [\frac 6 5, 2]$ and $\delta^* \in \left[0, \min \{\frac {\eta} 2, \frac 1 2\}\right)$  that satisfy
\begin{equation}\label{bandwidthffs}
\delta^* (k,m,n) \le C_{\delta^*} \min \{|\ck|^{-\xi}, |\cm|^{-\xi}, |\cn|^{-\xi}\}.
\end{equation}
  Then the FFS mixed terms satisfy the following estimate
\begin{align}\label{restrictedffs}
\left| \ip{\pD^{\ell}\wt B_s (\umm_f,\umm_f),\pD^{\ell} \wmm} \right|  &{\lesssim} \|\pD^{\frac 5 2 - \frac {3\xi} 4 - a +\ell_1  }  \umm_f \|_{L^2} \|\pD^{1 + a-\frac \xi 2}  \umm_f \|_{L^2}   \|\pD^{\ell_2}  \wmm \|_{L^2},
\end{align}
for  any $a\in(0, \frac 3 2 - \frac \xi 4)$  and $\ell_1\geq 0,\, \ell_2\in\R$ with $\ell_1+\ell_2=2\ell-1,$ where the implied constant depends on $\eta, \sL_1, \sL_2, \ell,\ell_1,\ell_2,a, \xi,C_{\delta^*}$.
\end{theorem}
A straightforward modification of  the proof of the previous theorem yields the following result on the level of linear potential vorticity.
\begin{corollary}\label{restrictedlpvC}
Under the assumptions of Theorem \ref{restrictedT},
 the following estimate  holds true
\begin{align}\label{restrictedlpv}
\left| \ip{\lpv \wt B_s (\umm_f,\umm_f),\lpv \wmm} \right|  &\lesssim \|\pD^{\frac 5 2 - \frac {3\xi} 4 - a  }  \umm_f \|_{L^2} \|\pD^{1 + a-\frac \xi 2}  \umm_f \|_{L^2}   \|\lpv  \wmm \|_{L^2},
\end{align}
for all $a\in(0, \frac 3 2 - \frac \xi 4),$ where the implied constant depends on $\eta, \sL_1, \sL_2, \xi, a, C_{\delta^*}$.
\end{corollary}
Theorem \ref{restrictedT} highlights the importance of the interaction coefficients, which will act as a smoothing factor in the convolution sums.  In fact, the parameter $3 - \frac{\xi} 2$ provides an upper bound for the exponent in the power law for the cardinality of  $\cN^{\textnormal {FFS}}$. Whereas we can already save some derivatives  via Lemma \ref{restrictedl},  further calculations on the interaction coefficients, as in  Lemma \ref{scof}, allow  us to lower the derivative cost even more.  In particular, the FFS estimate contains $2\ell+\frac 5 2-\frac {5 \xi} 4$ derivatives. As a consequence, our growth bound for $\delta^*$ is justified as follows. After the standard $L^2$ energy equality for \eqref{restrictedS} is obtained, an $L^\infty_t L^2_x$ estimate for $Q =\lpv\umm$ can only require
  an estimate  for $\umm_f$ in $L^2_t H^1_x$, but not in any higher regularity space. Therefore, in view of  Corollary \ref{restrictedlpvC}, which is essentially the case $\ell=1$, we
can only afford $\frac 3 2-\frac{5 \xi} 4 \leq 0$.

The remaining mixed interactions are covered in the following. 
\begin{theorem}\label{restrictedffT}
Let  $\ell\ge0$ and $\umm,\vmm \in H^{\ell+1}(\T;\R^4)$ be divergence and  zero-mean vector fields.    
 Consider any  $\xi \in [\frac 6 5, 2]$ and $\delta^* \in \left[0, \min \{\frac {\eta} 2, \frac 1 2\}\right)$  that satisfy \eqref{bandwidthffs}.
Then  the FSF mixed terms satisfy the following estimates
\begin{align}\label{restrictedfsf}
\left| \ip{\pD^{\ell} \wt B_f (\umm_f,\vmm_s),\pD^{\ell} \umm_f} \right|  & \lesssim \|\pD  \vmm_s \|_{L^2}   \|\pD^{\frac 3 2+\ell_1 - \frac {\xi} 4 -a }  \umm_f \|_{L^2}  \|\pD^{2\ell-\ell_1+a}  \umm_f \|_{L^2} \\\
& +\|\pD^{\ell_1 + 1}  \vmm_s \|_{L^2}   \|\pD^{  \frac 3 2 - \frac {\xi} 4 -a'  }  \umm_f \|_{L^2} \|\pD^{2\ell-\ell_1+a'}  \umm_f \|_{L^2}, \nonumber 
\end{align}
for  any $a, a'\in(0, \frac 3 2 - \frac \xi 4)$  and $\ell,\ell_1>0$ with $\ell_1<2\ell$, where the implied constant  depends on $\eta, \sL_1, \sL_2, \ell,\ell_1,a,a', \xi,C_{\delta^*}$. 
\end{theorem}

The difference of solutions between the original and the restricted system can be estimated   under appropriate lower bounds for the slow and fast bandwidths. The result  depends on the  ratio $\nu_R:= \frac {\nu_{max}}{\nu_{min}}$ with $\nu_R \in [1,\infty).$ 
\begin{theorem}\label{differenceT}
Let  $N>0$, $\xi \in [\frac 6 5, 2]$,  $\ell>\max\{3+\xi, \frac 9 2\}$ and $\ell' \in [1, \ell-2-\xi)$  be fixed. Consider two divergence free and  zero-mean  vector fields $\Umm_0,  \Ummt_0 \in H^{\ell}(\T;\R^4)$.  Let $\Umm,\Ummt$ be  solutions of the Boussinesq system \eqref{boussi0} and  the approximate system \eqref{restrictedS}, with initial data  $\Umm_0,  \Ummt_0$, respectively. Suppose, for positive constants $E_0, c_s, c_f$, that
\[
\| \Umm_0\|_{H^{{\ell}}}\leq E_0,\quad \|\Ummt_0\|_{H^{\ell}}\leq E_0,
\]
 and suppose 
\[
\delta^*(k,m,n) \ge c_s \min\{|\ck|^{-\xi},|\cm|^{-\xi},|\cn|^{-\xi}\},\quad \delta(k,m,n) \ge c_f \min\{|\ck|^{-1},|\cm|^{-1},|\cn|^{-1}\}.
\]
Then for all $\nu_1, \nu_2>0$ there exists a constant $C=C(c_f, c_s, \ell,\ell', \xi, \eta,\T,\nu_{max},\nu_R)$, which remains bounded as $\nu_{max}$ vanishes, and a time $T = T(E_0,\ell,  \eta,\T)$  such that
\[
\| \Umm - \Ummt\|_{H^{\ell'}}^2 \leq \| \Umm_0 - \Ummt_0\|_{H^{{\ell'}}}^2 + C N^{-2}\quad\text{for all}\quad t\in [0,T].
\] 
\end{theorem}
Theorem \ref{differenceT} is in line with the corresponding result from \cite{BCZNS}. An extra subtlety comes from the different  lower bounds for the bandwidths. This is nevertheless expected, due to the presence of more mixed interactions, as is shown by our arguments in  Appendix \ref{appendixlb}. 

\begin{remark}
A significant phenomenon in the exact resonance case is the dependence of error estimates on $\sL_1, \sL_2,$ \cite[Theorem 5.3]{BMN2}, \cite[Theorem 2]{GAL1}. In contrast to that, all our results, including the error estimates of Theorem \ref{differenceT}, depend smoothly on the domain parameters.
\end{remark}

\begin{remark}
We remark that Theorems \ref{globalexistenceT}, \ref{differenceT} continue to hold true, up to certain modifications, in the presence of a forcing term $F.$  In more detail, local in time regularity assumptions, like $F\in L^2 \left([T_0,T_0+1];H^{\ell}(\T)\right)$, for all $T_0>0$ and suitable $\ell>0$, would yield similar results, after taking into consideration  bounds for the relevant norms of $F$. In the exact resonance literature, results in that direction  have been obtained in  \cite{BMN2}, \cite{BMN1}, with \cite{GAL1} covering some different regularity assumptions on the forcing.
\end{remark}
\end{subsection}
\begin{subsection}{Literature review { and applicability of the model}}
A non-exhaustive list  of some previous works, concerning mainly the periodic case, follows. In the inviscid case with well-prepared data and a specific domain geometry, a rigorous study of convergence results for the asymptotic limit of \eqref{bous0} appears in \cite{BB1}. On the other hand, a detailed study of the limit equations in the resonant setting has been carried  in \cite{BMNZ}, in the absence of external forces, and \cite{BMN2},\cite{BMN1} in the forced case.  In this series of works, rigorous existence results were proved for the limit quasigeostrophic and ageostrophic systems, together with estimates of error in Sobolev spaces. An asymptotic description for the weak limit of solutions to the Boussinesq equations was given in  \cite{EM1},\cite{EM2}, including  the case of an axis of rotation different from $\e$ and that of arbitrary initial data.  In this series of works, the limiting arguments were carried in the cases of high stratification and finite rotation, and that of high stratification and rotation.  Moreover, the  fast-slow splitting  and interaction coefficients for the Boussinesq system, in the context of exact resonance,  are  extensively used. 

In \cite{GAL1}, global well-posedness of \eqref{bous0} was shown for initial data in $H^\ell (\T;\R^4)$, $\ell >1$, in the presence of a forcing term. Those results hold true for almost all domain ratios and arbitrary initial data.    The almost-periodic case, under stratification effects only, has been examined in \cite{IY}.  
As far as works related to multiscale asymptotic analysis are concerned, we mention \cite{WTW}, \cite{WW} and \cite{AOW}. Finally, in \cite{MJ},   singular three-scale limit methods have been  applied to \eqref{bous0} in the periodic case, with well-prepared initial data.

 We remark that, within \cite[Section 5]{BMN2}, a different concept of near resonance is presented. In more detail,  in the so-called quasiresonance framework of that work, the near resonance condition is expressed via the distance of a ratio $\eta$ to a distinguished value $\eta^*$, which corresponds to exact resonance. Then, the  obtained results are used in some small-divisor estimates.   Nevertheless,   the aforementioned analysis leads to results which feature a discontinuous dependence on the domain parameters, due to the methods used therein. 

  On the physical front, the important effect of near resonance  has been recognized and  analysed dating  back to \cite{BR} and \cite{ACN}.  In the particular case  of the forced rotating stratified Boussinesq system,   we mention  amongst other works \cite{SW}, \cite{SS}, \cite{OKS}, with the latter  highlighting the role of near resonance in energy transfer considerations.

{ The mathematical foundation of the near resonant approximation \eqref{restrictedS} that we study has several potential applications. First, the existing methodology of rotating, stratified turbulence simulations, \cite{SW}, can be applied to \eqref{restrictedS}, under the introduction of suitable forcing terms, and in the presence of fast-fast-fast (i.e. wave-wave-wave) and mixed (i.e. wave-vortical) interactions. Then, the study of energy exchanges between vortical and wave modes and the kinetic energy spectra for solutions to our system can shed new light in interaction regimes where QG-turbulence was the only alternative, up to now. Furthermore, recent numerical studies, \cite{PHW}, suggest that the near resonant approximation to fluid systems can lead to numerical models with greater computational efficiency, in particular having advantages in tackling the time-stepping issues in the case of a strong rotation or stratification.}

The remaining of this work is organized as follows. We first review our notational conventions in Section \ref{notation}, with Section \ref{preliminariesS}  containing some necessary preliminary results.  The restricted convolution results that will be used in the proofs of Theorems \ref{restrictedfT}, \ref{restrictedT} and \ref{restrictedffT} are proved in Section \ref{restricted:int:S}. Then,  Section \ref{countingS} contains estimates on the number of FFF, FFS and FSF near resonant interactions, via the reduction of lattice point counting to volume estimates.  Finally, Section \ref{proofS} contains the proofs for the main results stated in the Introduction.
\end{subsection}
\end{section}

\begin{section}{Notation}\label{notation}
We briefly summarize some of the notation and conventions used throughout the text.  First, a sum $\displaystyle \sum_{\substack{(k,m,n)\in{\Z^3}\\k+m+n=\vec 0}}$ will be simply denoted by $\skm.$  With the domain-adjusted
wavevector $\ck$ already defined in the  introduction, we further define $\ck_H := (\ck_1 ,\ck_2 )^\intercal,$ corresponding to the horizontal part of  $\ck,$  with $k_H$ reserved for the case $\sL_1=\sL_2=1.$ The same conventions carry over to any $m,n \in \Z^3$ occurring throughout the text.
We will write  $\ck_\eta : = (\ck_1,\ck_2, \eta k_3)^{\intercal},$ for the modification of a given vector $\ck\in\Z^3$ in the  vertical direction, so that the dispersion relation for \eqref{bous0} is given by $\omega_k = \frac {|k_\eta|} {|k|}$.  Moreover, we will write $\Delta_\eta$ for the modified Laplacian
\[
\Delta_\eta= \partial_{x_1}^2+\partial_{x_2}^2+ \eta^2 \partial_{x_3}^2.
\] As far as the viscosity and heat conductivity constants  in the dissipative terms are concerned, we define $\nu_{min}:=\min \{\nu_1, \nu_2\}$ and  $\nu_{max}:=\max \{\nu_1, \nu_2\}$. Finally, if  $\ck$ is a 3D domain-adjusted wavevector, then we define $\ck' :=( \ck_1 , \ck_2 , k_3 , 0)^\intercal$. 

 If the variables $A_1,\ldots,A_k$ are known, we will write $A \lesssim_{A_1,\ldots,A_k} B$ in order to denote an inequality of the form
\[
A \leq c B, \quad\text{with}\quad c:=c(A_1,\ldots A_k).
\] Unless otherwise stated, the implied constant in estimates of the aforementioned form will depend on $\eta$. The letter $C$ is reserved for constants, whose possible dependencies will be made explicit in each case.

The standard inner product in $L^2(\T;\C^4)$ will be denoted by
\[
\ip{f,g} =  \int_{\T} \cip{f(x)}{\overline g(x)} \rd x,
\]
with $\cip{f}{g}$ standing for the ordinary dot product for vectors.
The version of the Fourier transform that we use throughout this work is
\[
g_n :  =\dfrac{1}{|\T|}\int_{\T}g(x)\,e^{-\ri \cn \cdot x}\, \rd x\,,
\]
for $n\in \Z^3,$
so that $g(x)=\sum_n g_n e^{\ri \cn\cdot x}$. The
 corresponding Parseval's identity is given by:
\begin{equation}\label{parseval}
\ip{f(x),g(x)}=|\T|\sum_{n \in \Z^3 }    f_n \cdot \overline{g_n}.
\end{equation}

 Given a vector field $\umm\in L^2 (\T;\R^4),$ with Fourier coefficients $u_n,$ we write $u_n^{\sigma}:= \cip {u_n}{\overline{r_n^{\sigma}}}$, for $n\in\Z^3$ and a choice of sign $\sigma \in \{0,+,-\}.$ In a similar manner,  the Fourier coefficients of the potential vorticity $Q$ will be denoted by $Q_n$.  Finally, we recall that in the class of fields $\Umm:\T\to \R^4$ we have $\overline{U_{n}} = U_{-n}.$

Since the spherical coordinate system will be widely used in what follows, we reserve the notation $(\theta_k,\phi_k)$ for the polar and azimuthal angles of the wavevector $k\in\Z^3$, respectively. In addition, we set $c_k:= \cos \theta_k$, with a similar convention for $m,n.$ Then the dispersion relation can be expressed in a simpler manner as $\omega_k = \sqrt{1 + (\eta^2-1) c_k^2}$.

The dependence of the mixed and fast bandwidths $\delta^*,\delta,$ on the wavevectors will be explicitly given, unless otherwise stated.

Finally, we assume, without loss of generality,  that the velocity field has a zero mean. This holds true up to a translation of the velocity  by the mean drift, see e.g. \cite{BCZNS} for more details.
In the class of zero-mean fields $\Umm$ defined on $\T$, we have $\displaystyle\|\Umm\|_{H^{\ell}} = \sum_{k\in\Z^3\setminus\{\vec 0\}} |\ck|^{2\ell}|U_k|^2$.

\end{section}
\begin{section}{Preliminaries}\label{preliminariesS}
\begin{subsection}{The linear problem and the eigenvalues of $\Le$}
We recall some elementary facts on the spectrum of the  linear operator $\Le$. These will be used to further examine the fast and slow parts of a solution of { \eqref{restrictedS}}.
We begin with considering the linear initial value problem corresponding  to \eqref{boussi0}, ignoring the viscosity and heat effects, namely:
 \begin{align}\label{bous3}
&\partial_\tau \V  =   \Le\V,\quad\text{with divergence free}\quad\V(0) = \V_0 \in H^{\ell}(\T ; \C^4),
\end{align}
for some $\ell\geq 0$. In particular, we have the following.
\begin{proposition}
Let  $\V_0 \in H^{\ell}(\T ; \C^4)$ be divergence free with zero-mean. Then, a unique solution $\V \in C^{\infty}\left(\R; H^{\ell}(\T ; \C^4)\right)$  for the system \eqref{bous3} is given by $\V(\tau) = e^{\tau \Le} \V_0$.
\end{proposition}
 The claim  can be shown by arguing on the level of the Fourier transform, since $e^{\tau \Le}$ is a unitary operator. Also,  one can easily deduce that $e^{\tau \Le}$ commutes with  $(-\Delta)^\ell$, for ${ \ell\in \R}$, via the use of Fourier series.

 On the Fourier side, we have $\leray_k^\backprime = \ind - \frac{\ck\otimes \ck}{|\ck|^2}$, for $k\in\Z^3\setminus\{\vec 0\}.$ Thus, 
 \[
 \leray_k = \begin{pmatrix}
\leray_k^{\backprime}&0\\
0&1
\end{pmatrix}\quad\text{and}\quad
 \Le_k =\leray_k \begin{pmatrix}
\eta J & 0\\
0 & J
\end{pmatrix}\leray_k.
\]
We recall from \cite{EM2} that the eigenvalues of $\Le_k$ are  given by  0 with multiplicity 2 and
\[
 \pm \ri\omega_k = \pm \ri \sqrt{\dfrac {|\ck_H|^2 + \eta^2 k_3^2}  {|  \ck|^2}},\quad\text{for}\quad k\in\Z^3\setminus\{\vec 0\}.
\]
For eigenvectors, we introduce the following vectors depending on wavevector $k\in\Z^3$:
\begin{equation}\label{defeigenv1}
e_k^0:=\begin{cases} 
( -\ck_2,\ck_1,0  -\eta  k_3)^{\intercal},&\text{when}\quad{|\ck_H|\neq 0},\\
 (0,0,0,1)^{\intercal},&\text{else},
 \end{cases}
\end{equation}
and
\begin{equation}\label{defeigenv}
\alpha_k:=\begin{cases}\left( -\eta \ck_2 k_3,  \eta \ck_1 k_3, 0,  |\ck_H|^2\right)^{\intercal}+ \ri \omega_k  \left( \ck_1 k_3,\ck_2 k_3, - |\ck_H|^2,  0\right)^{\intercal},&\text{when }{|\ck_H|\neq 0}\\(\ri,1,0,0)^{\intercal},&\text{else}.\end{cases}
\end{equation}
  Then, the normalized eigenvectors for $\Le$, when $|\ck_H|\neq 0$, are given by
\begin{equation}\label{defeigenvs1}
\begin{cases}r^{00}_k=|\ck|^{-1}(\ck_1,\ck_2,k_3,0)^{\intercal},\quad&\text{for the zero eigenvalue,}\\
r^0_k = |\ck_\eta|^{-1} e_k^0,\quad&\text{for the zero eigenvalue,}\\
r^+_k=  (\sqrt{2}|\ck_H| |\ck_\eta|) ^{-1} \alpha_k,\quad &\text{for the $\omega_k$ eigenvalue},\\
r^-_k = (\sqrt{2}|\ck_H| |\ck_\eta|) ^{-1} \overline{\alpha_k},\quad&\text{for the -$\omega_k$ eigenvalue},
\end{cases}
\end{equation}
using the usual notation for the complex conjugate. On the other hand, for a purely vertical wavevector with $k_H=0$, we have the following normalized eigenvectors 
\begin{equation}\label{defeigenvs}
r_k^{00}=(0,0,1,0),\quad r_k^0=e_k^0,\quad r^+_k= \frac 1 {\sqrt 2} \alpha_k,\quad
 r^-_k=\overline{r^+_k}.
\end{equation} 
In both cases, the set of vectors $\{r^0_k, \rfp,\rfm \}$ forms an orthonormal basis of the subspace of $\C^4$  that corresponds to  incompressible vector fields with wavevector $k$, namely any $(v_1,v_2,v_3,v_4)^{\intercal}$ with $(v_1,v_2,v_3)\cdot\ck=0$. 

We briefly state some simple properties of { these eigenvectors}, taking into account the convolution condition as well.
\begin{lemma}\label{properties}
Let $k,m,n\in \Z^3$, then the following properties hold true:
\begin{enumerate}
\item $r^{\pm}_{-k} = r^{\pm}_{k}$\label{i1}
\item $e^0_n + e^0_k + e^0_m = \vec 0,$ when $k+m+n=\vec 0$.
\end{enumerate}
\end{lemma}
\begin{proof}
First,  the presence of an even number of components of $k$ in $\alpha$ immediately implies the first statement.   The last statement follows from the linearity of the eigenvector, combined with the convolution condition.
\end{proof}

We remark that \[
\sum_{k\in\Z^3}\sum_{\sigma\in\{0,\pm\}}\leray_k^{\sigma} \leray \umm = \sum_{k\in\Z^3}\sum_{\sigma\in\{0,\pm\}} \leray \leray_k^\sigma \umm = \leray \umm,
\]
where $\leray$ stands for the trivially extended Leray projection, directly following from
 the incompressibility of the basis.  Using the eigendecomposition formalism,  the action of the evolution operator on a divergence-free vector field  $\Umm \in L^2(\T,\R^4)$ can then be expressed as follows
\begin{equation}\label{propexp}
e^{\tau \Le} \Umm = \sum_{k\in\Z^3}\sum_{\sigma\in\{0,\pm\}} e^{\ri \omega_k^\sigma \tau } \leray_k^\sigma \Umm.
\end{equation}
\begin{remark}\label{urealR}
In the class of real-valued vector fields $\umm$ in $\T$, we have:
\begin{equation*}
\overline{\cip {u_k}{\overline{r}_k^{\sigma}}  r_k^{\sigma}} = \cip {u_{-k}}{{r}_{-k}^{\sigma}} {r_{-k}^{-\sigma}},
\end{equation*}
for all $\sigma \in \{\pm,0\}$, as an immediate consequence of Lemma \ref{properties}.
\end{remark}
The  dependence of the linear potential vorticity $Q$ on  the slow part of $\Umm\in H^1 (\T;\R^4)$ can be made explicit,  thanks to the fact that  the Fourier symbol of  operator $\lpv$ is simply $e_k^0$,  as given in \eqref{defeigenv1}. Therefore,
\begin{equation}\label{eqlpvslow}
\lpv \Umm  = \ri\sum_{k\in\Z^3} |\ck_\eta|\cip{U_k}{r_k^0}e^{\ri\ck\cdot x}= \lpv\Umm_s.
\end{equation} 
The conjugation on $r_k^0\in\R^4$ is omitted for brevity. Then orthogonality of the basis vectors $r_k^0$ and $r_{k}^{\pm}$ readily  implies the following.
\begin{lemma}\label{orthogonalityL}
Let $\ell\in\R$ and $\Umm, \Vmm \in H^{\ell}(\T;\R^4)$ be  divergence-free and zero-mean vector fields. Then the following statements hold true
\begin{enumerate}[i)]
\item $\ip{\lpv\Umm,\lpv\Vmm}= \ip{\Umm_s,\dsn{}\Vmm_s},$
\item $\ip{\Umm_s,\ds{\ell}\Vmm_f}=0$.
\end{enumerate}
\end{lemma}
\begin{proof}
In order to prove the first statement, we expand using \eqref{parseval} and \eqref{eqlpvslow}:
\[
\ip{\lpv\Umm,\lpv\Vmm}  
 = |\T|\sum_{k\in\Z^3\setminus\{\vec 0\}}|\ck_\eta|^2 \cip{U_k}{r_k^0}\overline{\cip{V_k
}{r_k^0}}.
\]
We use  \eqref{parseval} once more for the second statement
\[
\ip{\Umm_s,\ds{\ell}\Umm_f} = |\T|\sum_{k\in\Z^3\setminus\{\vec 0\}} |\ck|^{2\ell} \cip{U_k}{r_k^0}\left[ \cip{r_k^0}{\overline{r_k^+}} \overline{\cip{V_k}{\overline{r_k^+}	}} +  \cip{r_k^0}{\overline{r_k^-}}\overline{\cip{V_k}{\overline {r_k^-}}}\right]=0.
\]
\end{proof}
\end{subsection}
\begin{subsection}{The restricted bilinear and elliptic operators}\label{restrictedLB}
The following result will prove crucial in obtaining the standard $L^2$ identity for { \eqref{restrictedS}}.
\begin{lemma}\label{L2l}
Let $\Umm, \Vmm \in H^1(\T;\R^4)$ be divergence-free and zero-mean vector fields. Then, for any $\tau\in\R$, the following properties hold true
\begin{enumerate}[i)]
\item$\wt B (\tau,\Umm,\Vmm) \in \R^3$\label{enum1}
\item$\ip{\wt B_s (\tau,\Umm_s,\Vmm_s),\Vmm} = 0$\label{enum2}
\item$\ip{\wt B_f (\tau,\Umm_s,\Vmm_f),\Vmm} =\ip{\wt B_f(\tau,\Umm_f,\Vmm_f),\Vmm}= 0$\label{enum3}
\item $\ip{\wt B_f (\tau,\Umm_f,\Vmm_s),\Vmm} =- \ip{\wt B_s(\tau,\Umm_f,\Vmm_f),\Vmm}$.\label{enum4}
\end{enumerate}
\end{lemma}
\begin{proof}
It suffices to prove the Lemma in the case $\tau=0.$ We prove \eqref{enum1} for the slow part $B_s$ only. In particular,  taking the complex conjugate yields
\begin{align*}
&\overline{\wt B_s (\Umm,\Vmm)}
= \skm   \overline{B_{kmn}^{000}(\Umm,\Vmm)} +\sum_{\substack{k,m,n;conv}}\sum_{\sigma_1 \sigma_2 < 0}  \overline{B_{kmn}^{\sigma_1 \sigma_2 0}(\Umm,\Vmm)}\ind_{\cN^{\textnormal {FFS}}}.\end{align*}
For the SSS term, we change $(k,m)\to -(k,m)$ in the sum, so that we get
\begin{align*}
\skm \overline{B_{kmn}^{000}(\Umm,\Vmm)}&=-\skm   \cip{U_{-k}}{r_k^0} \cip{r_k^0}{\ri \cm'} \cip{V_{-m}}{r_m^0}  \cip{r_m^0}{r_{-n}^0} r_{-n}^0\\
& =\skm {B_{kmn}^{000}(\Umm,\Vmm)}.
\end{align*}
For the mixed FFS term, we take into account that $r_{k}^{\sigma_1}=\overline{r_k^{\sigma_2}}$ and $r_k^{-\sigma_1} = r_k^{\sigma_2},$ for nonzero $\sigma_1\neq \sigma_2$, and change $(k,m)\to -(k,m)$ in the sum, so that
\begin{align*}
\skm \overline{B_{kmn}^{\sigma_1 \sigma_2 0}(\Umm,\Vmm)} &= - \skm \cip{U_{-k}}{r_k^{\sigma_1}} \cip{r_k^{-\sigma_1}}{\ri \cm'} \cip{V_{-m}}{r_m^{\sigma_2}}  \cip{r_m^{-\sigma_2}}{r_{-n}^0} r_{-n}^0\\
&= \skm    \cip{U_{k}}{\overline{r_k^{\sigma_2}}} \cip{r_k^{\sigma_2}}{\ri \cm'} \cip{V_{m}}{\overline{r_m^{\sigma_1}}}  \cip{r_m^{\sigma_1}}{r_{-n}^0} r_{-n}^0\\
&=\skm B_{kmn}^{\sigma_2 \sigma_1 0}(\Umm,\Vmm).
\end{align*}
Since $\ind_{\cN^{\textnormal {FFS}}}$ is an even function, the first claim of the Lemma follows.

As \eqref{enum2}, \eqref{enum3} and \eqref{enum4} all follow the same reasoning, we only prove the last one.  In that direction, we recall that 
 \[
\ind_{\cN^{\textnormal {FFS}}}(k,m,n) = \ind_{\cN^{\textnormal {FSF}}}(k,n,m),
\]
directly from the definition of the mixed sets. Then, we expand using the incompressibility of $r_k^{\pm}$:
\begin{align*}
\ip{\wt B_s (\Umm_f,\Vmm_f),\Vmm} &=  -\skm \cip{U_{k}}{r_k^-} \cip{r_k^+}{\ri \cn'} \cip{V_{m}}{r_m^+}  \cip{r_m^-}{r_{n}^0} \cip{V_{n}}{r_n^0}\ind_{\cN^{\textnormal {FFS}}}(k,m,n) \\
&- \skm \cip{U_{k}}{r_k^+} \cip{r_k^-}{\ri \cn'} \cip{V_{m}}{r_m^-}  \cip{r_m^+}{r_{n}^0} \cip{V_{n}}{r_n^0}  \ind_{\cN^{\textnormal {FFS}}}(k,m,n) 
\end{align*}
so that changing between $m,n$ in the sum yields \eqref{enum4}.
\end{proof}
The fact that 
 \begin{align*}
\wt B(\Umm,\Vmm) &= \wt B_s(\Umm,\Vmm)+ \wt B_f(\Umm,\Vmm) \\
&= \wt B_s(\Umm_s,\Vmm_s)+\wt B_s(\Umm_f,\Vmm_f)+ \wt B_f(\Umm_f,\Vmm_f)+ \wt B_f(\Umm_s,\Vmm_f)+ \wt B_f(\Umm_f,\Vmm_s)
\end{align*}
immediately implies the following.
\begin{corollary}\label{L2l0}
Under the assumptions of Lemma \ref{L2l} we have
\[
\ip{\wt B(\Umm,\Vmm),\Vmm}=0.
\]
\end{corollary}
We now present the expanded form of the Laplacian terms so that the motivation for
introducing $\wt \nu_{11}, \wt \nu_{22}$ is highlighted.  In particular,
the slow and fast projections of the original Laplacian terms $A\umm$ in \eqref{boussmod0}, respectively, are (recall again $r_k^0\in\R^4)$
\begin{itemize}
\item $ -(\boldsymbol{\nu}\Delta\umm)_s=\displaystyle  \sum_{k\in\Z^3\setminus\{\vec 0\}}\sum_{\sigma \in \{\pm,0\}} e^{\ri (\ck\cdot x +  \sigma\omega_k \tau)} |\ck|^2  \cip{\boldsymbol{\nu} r_k^{\sigma}}{r_k^0} \cip{u_k}{\overline {r_k^\sigma}}  r_k^0 $
\item $ -(\boldsymbol{\nu}\Delta\umm)_f=\displaystyle  \sum_{k\in\Z^3\setminus\{\vec 0\}}\sum_{\sigma \in \{\pm,0\}}\sum_{\sigma_1\in\{\pm\}} e^{\ri (\ck\cdot x -{\sigma_1} \omega_k \tau +{\sigma} \omega_k \tau)} |\ck|^2\cip{\boldsymbol{\nu} r_k^{\sigma}}{\overline{r_k^{\sigma_1}}}  \cip{u_k}{\overline {r_k^{\sigma}}}  r_k^{\sigma_1}$,
\end{itemize}
whereas their respective approximations are given by:
\begin{itemize}
\item\hfill $\displaystyle -\wt\nu_{11} \Delta \umm_s= \sum_{k\in\Z^3\setminus\{\vec 0\}} e^{\ri \ck\cdot x } |\ck|^2  \cip{\boldsymbol{\nu} r_k^0}{r_k^0} \cip{u_k}{ {r_k^0}}  r_k^0  =\sum_{k\in\Z^3\setminus\{\vec 0\}}  |\ck|^2  \cip{\boldsymbol{\nu} r_k^0}{r_k^0}  P_k^0\umm $\hfill\refstepcounter{equation}\textup{(\theequation)\label{visco1}}
\item $\displaystyle - \wt\nu_{22} \Delta \umm_f= \sum_{k\in\Z^3\setminus\{\vec 0\}}\sum_{\sigma_1 \in \{\pm \}} e^{\ri \ck\cdot x  } |\ck|^2\cip{\boldsymbol{\nu} r_k^{\sigma_1}}{\overline{r_k^{\sigma_1}}}  \cip{u_k}{\overline {r_k^{\sigma_1}}}  r_k^{\sigma_1}\\ = \sum_{k\in\Z^3\setminus\{\vec 0\}}\sum_{\sigma_1 \in \{\pm \}}   |\ck|^2\cip{\boldsymbol{\nu} r_k^{\sigma_1}}{\overline{r_k^{\sigma_1}}} P_k^{\sigma_1}\umm$.\hfill\refstepcounter{equation}\textup{(\theequation)\label{visco2}}.
\end{itemize}
Thus, the effect of the fast time $\tau$ is not present in this part of our approximation.  Observe that $\wt\nu_{11}$ and $\wt\nu_{12}$ are scalar pseudodifferential operators of degree 0, hence commute with Laplacian, slow and fast projections, the evolution operator $e^{\tau\Le}$ and the linear PV operator $\lpv$. In particular,
 we have 
 \begin{equation}\label{lpvlaplacian}
 -\lpv\left(\wt\nu_{11} \Delta \umm_s\right)= - \wt \nu_{11}\Delta Q= \sum_{k\in\Z^3\setminus\{\vec 0\}} e^{\ri \ck\cdot x } |\ck|^2  \cip{\boldsymbol{\nu} r_k^0}{r_k^0} Q_k.
 \end{equation} Finally, we record the following ellipticity property of the modified Laplacian operators.
\begin{lemma}\label{ellipticity}
Let $\ell\in\R$, $\Ummt\in H^{2\ell+2}(\T;\R^4)$, and $\nu_1,\nu_2 \geq 0$, then the following  estimates hold true 
\begin{enumerate}[i)]
\item $\nu_{min}  \| \Ummt_s \|_{H^{\ell +1}}^2 \leq -\ip{\wt{\nu}_{11}\Delta  \Ummt_s, \ds{\ell} \Ummt_s}$,
\item  $\nu_{min}  \| \Ummt_f \|_{H^{\ell +1}}^2 \leq -\ip{\wt{\nu}_{22}\Delta  \Ummt_f, \ds{\ell} \Ummt_f}$,
\item $\nu_{min}   \| \wt Q \|_{H^{\ell +1}}^2 \leq -\ip{\wt{\nu}_{11}\Delta \wt Q, \ds{\ell}\wt Q}$,
\end{enumerate}
where $\wt Q=\lpv\Ummt.$
 \end{lemma} 
\begin{proof}  
The first and  third statement follow from \eqref{parseval}, \eqref{lpvlaplacian} and the fact that  the Fourier symbol of $\wt \nu_{11}$ is
\[\cip{\boldsymbol{\nu} r_k^{0}} {r_k^{0}}=\frac{{\nu_1|\ck_H|^2 + \nu_2\eta^2|k_3|^2}} {|\ck_\eta|^{2} }.\]  
The second statement follows similarly, after taking into account the orthogonality $r_k^{+}\cdot \overline{r_k^-}=0$ and the form of the Fourier symbol of $\wt{\nu}_{22}$
\[
\cip{\boldsymbol{\nu} r_k^{+}} {\overline{r_k^{+}}}=\cip{\boldsymbol{\nu} r_k^{-}} {\overline{r_k^{-}}}=\frac{\nu_1 \eta^2|k_3|^2 + 2(\nu_2+\nu_1)|\ck_H|^2} {2|\ck_\eta|^2}.
\]
\end{proof}
\end{subsection}
\end{section}
\begin{section}{Restricted Convolution}\label{restricted:int:S}
In this section we present two restricted convolution lemmas, one to be applied to FFF interactions and the other to be applied to FFS and FSF interactions, respectively. 
\subsection{{Restricted convolution under full symmetry}}
We recall the following version of \cite[Lemma 3.1]{BCZNS}, which we will use in the proof of the FFF estimate Theorem \ref{restrictedfT}.
\begin{lemma}\label{restricted:L}
 Let  $\cN \subset (\Z^3\setminus\{\vec 0\})^3$, with $\ind_{\cN}(\cdot,\cdot,\cdot)$ symmetric with respect to any permutation of its arguments, and
\[
\cN_0:=\big\{(k,m,n)\in\cN:  |\cm| \le |\ck| \le |\cn| \big\}.
\] 
If there exist a constant $\beta\in\big[0,  3\big]$ and a constant $C_{\cN}$ so that the counting condition

\begin{equation}\label{ci0}
   \sum_{\substack{k \in \Z^3 }} \ind_{\cN_0}(k,-n-k,n)    \le C_{\cN} \, |\cn|^{\beta},\qquad \forall n \in \Z^3\setminus\{\vec 0\}
\end{equation}
holds, then  for zero-mean $\umm,\vmm\in H^{\beta\over2}(\T)$ and $\wmm\in L^2(\T)$,  
\begin{equation}\label{restricted:f}
\begin{aligned}&\skm  | u_n|\,| v_k|\,| w_m |  \,\ind_\cN( n, k, m)
 \lesssim   \left( \|\umm\|_{L^2}\,\| \vmm\|_{H^{\frac {\beta} 2}}+ \|\umm \|_{H^{\frac {\beta} 2}}\,\|\vmm\|_{L^2} \right)\|\wmm\|_{L^2},
 \end{aligned}
\end{equation} 
with the implied constant depending on $C_{\cN}, \T.$
\end{lemma}

\begin{subsection}{{Restricted convolution under reduced symmetry}}
We now move on to the  following crucial result for controlling the mixed FFS and FSF energy estimates. The main difference to   the related   rotating Navier-Stokes results, \cite[Lemma 3.1]{BMN3} and \cite[Lemma 3.1]{BCZNS} for FFF interactions, is a reduced symmetry. The  proof is  still based on a Paley-Littlewood decomposition, under our restricted interaction setting. 

In more detail, in both cases of  FSF and FFS interactions, there is a permutation  symmetry  only between two  fast modes, e.g. $k$ and $n$. Thus, without loss of generality, we assume that $|\ck|\leq |\cn|$, and proceed to a Paley-Littlewood decomposition for $k$ in annuli of the form
\[
\mathcal A_i := \{k \in \Z: 2^{i-1} \leq |\ck|< 2^{i} \}.
\] If we further restrict $|\cm|\leq |\ck|, $ then the proof of the trilinear estimate in \cite[Lemma 3.1]{BCZNS} goes through unchanged, with $ k \in \Ai$ implying that $n \in \Ai \cup \Aii.$ However, we must also consider the case $|\cm|>|\ck|,$ a condition which does not provide us with an upper bound for $|\cn|,$  hence explaining the difference between  \eqref{ci0} and \eqref{ci} -- also see  the corresponding condition in  \cite[Lemma 3.1]{BMN3}. 

In what follows, we will  use the following  lemma concerning the sum of a sequence of weighted sums of Fourier coefficients that are restricted outside a geometric sequence of balls.
\begin{lemma}\label{prestrictedl}
Let { $\boldsymbol u \in H^{\ell}(\T;\R)$, $\ell>0$, with Fourier coefficients  $u_n.$} Moreover,  define set $\mathcal B_i = \{n \in \Z^3 :  |\cn|<2^i \}$. Then, the following inequality holds true  for all $\ell>0$:
\[
\sum_{{ i\in\N}}\sum_{n\in \mathcal B_i^{\mathsf{c}}} 
 2^{2i\ell} |u_n|^2 \lesssim_{\ell}{\|{\boldsymbol u}\|}_{H^{\ell}}^2.
\]
\end{lemma}
In fact, a more delicate relation of the form 
\[
\left\| \left( \sum_{ i\in\N}\sum_{n\in \mathcal B_i^{\mathsf{c}}}  2^{2i\ell} |u_n|^2 \right)^{\frac 1 2} \right \|_{\ell^p}   \sim  \| \ds{\frac \ell 2} \umm\|_{L^p}
\]
holds true for $p\in(1,\infty)$ and $\ell>0$, see e.g. \cite{KTV}. Nevertheless, we present a  proof of the weaker statement that we later need, due to its simplicity.
\begin{proof}
First, we note that $i\leq \log_2 |\cn|$, whereas the geometric sum satisfies
\[
\sum_{i=0}^{{ \lceil \log_2 |\cn| \rceil} }  2^{i\ell} { =} \dfrac {1- 2^{\ell({{ \lceil \log_2 |\cn| \rceil} }  + 1)}}{1- 2^{\ell}}\approx_{\ell} |\cn|^\ell.
\]
Thus, 
\[
{\sum_{ i\in\N}}\sum_{n\in \mathcal B_i^{\mathsf{c}}}  2^{2i\ell} |u_n|^2 \leq   \sum_{n\in \mathcal B_i^{\mathsf{c}}}|u_n|^2 \sum_{i=0}^{{{ \lceil \log_2 |\cn| \rceil} } }  2^{2i\ell}  \lesssim_{\ell}\sum_{n\in \mathcal B_i^{\mathsf{c}}} |\cn|^{2\ell} |u_n|^2 \lesssim \| {\umm}\|_{H^{\ell}}^2,
\]
{ with the exchange of sums justified by our regularity assumptions on $\umm,$ which implies the absolute converge of the double sum in any order and its commutability.}
\end{proof}
Then, the following  lemma will allow us to control the crucial mixed trilinear term  occurring in  the energy estimates.
\begin{lemma}\label{restrictedl}
Let  $\ind_{\cN^{'}} ( \cdot, \cdot, \cdot )$ be the characteristic of a subset of $(\Z^3\setminus\{\vec 0\})^3$, which is symmetric with respect to permutations in its first two arguments. Moreover, let $\Ai$ be the  annuli
 \[
 \Ai=\big\{k\in\Z^3\,\big|\,2^{i-1}\le | \ck| < 2^i\big\},\quad\forall i\in\Z^+,
 \] 
 and suppose
\begin{equation}\label{ci}
\sum_{\substack{\ck \in \Ai\atop {|\ck|\le |\cn|  }}} \ind_{\cN^{'}} (k,-n-k,n)    \le  C_0
 2^{\mu  i},\qquad \forall i \in \Z^+,\qquad \forall n\in\Z^3,
\end{equation}
for some $C_0, \mu\ge0,$ not depending on $i$ and $n$. Then, for every 
\[
{a\in(0,\frac \mu 2 )}
\]
 {and  for every  zero-mean
\[
 {\umm\in H^{\frac \mu 2 -a}(\T) ,\quad \vmm\in H^{a}(\T),\quad\wmm\in L^2(\T)},
 \]} we have
\begin{equation}\label{restricted}
\sum_{k,m,n;conv} \ind_{\cN^{'}} (k,-n-k,n)  |  u_k  v_n  w_m |   \lesssim  {\| \umm\|_{H^{\frac \mu 2 - a}}\|\vmm \|_{H^a}  \| \wmm \|_{L^2}  }.
\end{equation}
\end{lemma}
\begin{remark}\label{restrictedr}
Lemma \ref{restrictedl} can be generalized  to sets ${\cN^{'}}\subset(\Z^3\setminus\{\vec 0\})^3$, whose characteristic function possesses a permutation symmetry between any two arguments.
\end{remark}
\begin{proof}
First, we set $\wt \cN:= \cN' \cap \{(k,-n-k,n) \in(\Z^3\setminus \{\vec 0\})^3 : |\ck| \leq |\cn| \}$.  Then, due to the symmetry of $\ind_{\cN^{'}}$ in the variables $k$ and $n$, we have
\begin{align*}
\skm   | v_n u_k  w_m |\ind_{\cN'}{(k,-n-k,n) }  \leq   \skm (  | v_n u_k| +  | v_k u_n|) |  w_m | \ind_{\wt \cN}{(k,-n-k,n) }. 
\end{align*}
We will use a Paley-Littlewood type dyadic decomposition  in order to estimate the first term on the right, expanding the sum  with respect to $k$,  and the full result will follow via reversing the roles of $k$ and $n$. 
In that direction, we have
\[
 \skm    |  u_k   w_m   v_n |\ind_{\wt\cN}( k,-n-k,n)  = \sum_{i\in\Z^+} \sum_{\substack{k\in \Ai }} \sum_{\substack{n  \in \mathcal B_{i-1}^{\mathsf{c}} }}     |  u_k   w_m   v_n | \ind_{\wt\cN}(k,-n-k,n) .
\]
For the sake of brevity, we suppress the input arguments to the characteristic function in the rest of this proof.  Switching the order of summation between $k$ and $n$ and using the Cauchy-Schwarz inequality in the $n$-variable  we have
\begin{align} \label{form1}
 &  
 \sum_{i\in\Z^+} \sum_{\substack{k\in \Ai }} \sum_{\substack{n  \in \mathcal B_{i-1}^{\mathsf{c}} }}      |  u_k   w_{-n-k}   v_n |\ind_{\wt\cN}\\ \nonumber
 & \lesssim  \sum_{i \in \Z^+}\left( 2^{2a i}\sum_{\substack{ n \in \mathcal B_{i-1}^{\mathsf{c}} }} |  v_n |^2  \right)^{1/2}
\left[2^{-2a i}  \sum_{\substack{ n \in \mathcal B_{i-1}^{\mathsf{c}} }}\left(\sum_{ k \in \Ai  }    |  u_k   w_{-n - k}  | \ind_{\wt\cN} \right)^2 \right]^{1/2}.
\end{align}
 Then,  using \eqref{ci} and the Cauchy-Schwarz inequality in the $k$-variable, we obtain
\begin{align*} 
& 
\left[\sum_{\substack{ n \in \mathcal B_{i-1}^{\mathsf{c}} }} \left(\sum_{k\in\Ai} |  u_k   w_{-n - k}  |  \ind_{\wt\cN}   \right)^2 \right]^{1/2} \\
& \lesssim \left[\sum_{\substack{ n \in \mathcal B_{i-1}^{\mathsf{c}} }} \left(\sum_{k\in\Ai} |  u_k   w_{-n - k}  |^{2} \ind_{\wt\cN}  \right) \right]^{1/2}\sup_{n \in \mathcal B_{i-1}^{\mathsf{c}} } \left(\sum_{k\in\Ai} \ind_{\wt\cN}    \right)^{1/2} \lesssim 2^{\frac{\mu i} 2} \left[\sum_{\substack{ n \in \mathcal B_{i-1}^{\mathsf{c}} }} \sum_{k\in\Ai}  |  u_k   w_{-n - k}  |^{2}  {}\right]^{1/2}.
\end{align*}
Finally, we insert the last estimate into \eqref{form1} and combine it with the Cauchy-Schwarz inequality, in order to derive that
\begin{align*}
 &  \skm   |  u_k   w_{-n-k}   v_n | \ind_{\wt\cN}( k,-n-k,n) \\ \nonumber
 & \lesssim  \left( \sum_{i \in \Z^+} \sum_{\substack{ n \in \mathcal B_{i-1}^{\mathsf{c}} }} 2^{2a i }|  v_n |^2  \right)^{1/2}
  \left[ \sum_{i \in \Z^+}\sum_{\substack{ n \in \mathcal B_{i-1}^{\mathsf{c}} }} \sum_{k\in\Ai}2^{( \mu  - 2a) i}  | u_k   w_{-n - k}  |^{2}  {}    \right]^{1/2}\\
   & \lesssim  \left( \sum_{i \in \Z ^+} \sum_{\substack{ n \in \mathcal B_{i-1}^{\mathsf{c}} }} 2^{2a i }|  v_n |^2  \right)^{1/2}
  \left[ \sum_{i \in \Z ^+}\sum_{k \in \Ai }\left(  2^{(\mu-2a) i}  |  u_k|^2  {}   \sum_{\substack{ n \in \mathcal B_{i-1}^{\mathsf{c}} }}   {  |  w_{-n - k}  |^{2}}\right) \right]^{1/2},
\end{align*}
after switching the  sums containing $k$ and $n$. {By Parseval's identity, we have uniform bound $\sum_{\substack{ n \in \mathcal B_{i-1}^{\mathsf{c}} }}   {  |  w_{-n - k}  |^{2}}\lesssim \|\wmm\|_{L^2}^2$.} Thus,  \eqref{restricted}  follows from Lemma \ref{prestrictedl}.
\end{proof}
\end{subsection}
\end{section}
\begin{section}{Counting near resonant interactions}\label{countingS}
 We will address the counting of mixed FSF and FFS interactions in a short section before the  more involved case of FFF interactions.
\begin{subsection}{Counting mixed interactions}
 By fixing one of the fast modes in the mixed interaction set and limiting the length of the other fast mode, 
 the counting problem is then reduced   to a corresponding volume integral estimate, up to a lower order error term. The volume integral
 that comes up is expressed in rescaled spherical coordinates,  with a modified radial component. 
Crucially, the range of the polar  angle variable is restricted due to the near resonance condition.
We first  prove the volume estimates for the mixed sets, with  the reduction of lattice point counting to the volume following in Corollary \ref{mixedvolumec}. Finally, we note that the next Lemma is posed in terms of the FSF set, but still applies to the FFS one.
\begin{lemma}\label{mixedvolumel}
Let $n\in \Z^3\setminus \{\vec 0\},\delta^*\in\left[0,\min\{\frac \eta 2, \frac 1 2\}\right)$  and $M>0$,   be  fixed and consider  the set 
 \[
V_1= \left\{k\in \R^3\setminus \{\vec 0\}: 
\left|\dfrac {|\ck_\eta|} {|\ck|} -  \dfrac {|\cn_\eta|} {|\cn|}\right| \leq \delta^* \quad\text{and}\quad |\ck| \leq M \right\}.
\] Then, the following estimate holds true:
\[
vol(V_1) \leq  C(\eta)\sL_1 \sL_2 \sqrt{\delta^*}M^3,
\]
with $C$ independent of $\sL_1, \sL_2, M, n.$
\end{lemma}

\begin{proof}
  First, note the volume element scaling property  $\sL_1 \sL_2 \rd \ck =  \rd k.$ Thus,  it suffices to prove the estimate for $\sL_1 = \sL_2 =1.$ We introduce modified spherical coordinates for the wavevector $k,$ $\left(\lambda_k',  \theta_k, \phi_k \right)  \in [ 0, 1] \times [0, \pi] \times  [0, 2\pi]$, with a modified radial component $\lambda_k':=\frac {|k|}{M}.$  In particular,  we have
\[
 k =M\lambda_k' (\sin\theta_k \cos \phi_k, \sin \theta_k \sin \phi_k, \cos \theta_k),\quad\text{for}\quad k\in\Z^3.
\]
The near resonance condition in the mixed context implies the following range for $ c_k:=\cos\theta_k$
 \begin{equation}\label{range}
 (\eta^2 - 1) c_k^2\in \left[ (\omega_n - \delta^*)^2-1, (\omega_n  + \delta^*)^2 - 1  \right],
 \end{equation}
  after using the definition of the dispersion relation directly.  It immediately follows that $|c_k|$ is constrained to an interval $I$, with length    $|I| { \lesssim \sqrt{\frac{ \omega_n \delta^*}{|\eta^2 -1|}}}\lesssim_{\eta} \sqrt{\delta^*},$ after using the trivial upper bound on  $\omega_n$.  Then, using the change of variable $\frac {|k|}{M} = \lambda_k'$ and  also substituting $\cos\theta_k$ instead of $\theta_k$, we have
\begin{align*}
vol(V_1) =  M^3 \int_{0}^{1}\int_0^{2\pi} \int_0^{\pi} \ind_{V_1} \sin\theta_k   \rd\theta_k \rd\phi_k \rd \lambda_k' &\leq  M^3 \int_{0}^{1}\int_0^{2\pi} \int_{I}     \rd c_k d\phi_k  \rd \lambda_k'\\
 \leq  C(\eta)   \sqrt{\delta^*} M ^3,
\end{align*}
 via Fubini's Theorem.
\end{proof}
The following Corollary provides an estimate on the number of mixed interactions, if one of the fast wavenumbers is fixed.
\begin{corollary}\label{mixedvolumec}
The number of integers in  the set $V_1$ from Lemma \ref{mixedvolumel} satisfies
\[
\sum_{k\in\Z^3} \ind_{V_1}(k) \leq C\left(\sL_1\sL_2\sqrt{\delta^*} M^3 + \sL_1\sL_2 M^2 + (\sL_1+\sL_2) M\right),
\] for a constant $C$ that is independent of $\sL_1, \sL_2, M, n$. 
\end{corollary}
\begin{proof}
We first fix $k_H \in \Z^2\setminus\{\vec 0\}$. Then, defining $\wt F :\R^3\setminus\{\vec 0\} \to \R$ with $
\wt F (k_3) := \dfrac {|\ck_\eta|} {|\ck|} -  \dfrac {|\cn_\eta|} {|\cn|},$
we have
\[
\dfrac {\partial \wt F} {\partial k_3} = (\eta^2 - 1) \dfrac {k_3 |\ck_H|^2}{|\ck_\eta| |\ck|^3}. 
\]
Since $k_H\neq \vec 0,$  there are at most two intervals of  monotonicity for $\wt F.$ On each such interval, the number of  $k_3\in\Z$ can be estimated by $\int_{\R}\ind_{V_1} (k_H, k_3) \rd k_3 +1$, thus:
\[
\sum_{k_3 \in \Z} \ind_{V_1}(k_H, k_3) \leq 2 + \int_{\R} \ind_{V_1}(k_H, k_3) \rd k_3.
\]
The limits of the last integral depend linearly on $|\ck_H|$, as \eqref{range} holds true.
Hence, we sum over $\ck_H$, with $|\ck_H| \leq M^2$, so that estimating the left Riemann sums with the corresponding integrals, via \cite[Lemma A.2]{BCZNS}, yields
\[
\sum_{k\in \Z^3} \ind_{V_1} (k) \leq C\left(\sL_1\sL_2 M^2+ (\sL_1+\sL_2)M + vol(V_1)\right).
\]
\end{proof}
\end{subsection}
\begin{subsection}{Initial reduction to two-dimensional counting for FFF  interactions}
In order to estimate the number of FFF modes, under the symmetry assumptions of Lemma \ref{restricted:L},  we
 first  define the function $F_{n,\sigma_1,\sigma_2} :\R^3\setminus\{\vec 0, -n\}\to\R$ with:
\[
F_{n,\sigma_1,\sigma_2}(k) =   \dfrac {|\cm_\eta|} {|\cm|} + \sigma_1 \dfrac {|\cn_\eta|} {|\cn|}+\sigma_2 \dfrac {|\ck_\eta|} {|\ck|}  ,
\]
for fixed $n\in \Z^3 \setminus\{\vec 0\}$ and $\sigma_1,\sigma_2 \in\{\pm\}.$ In what follows, we  just write $F$ for the sake of brevity. In addition, since any wavevector $k \in \cN_0$ satisfies $\frac{|\cn|} 2 \leq |\ck|\leq |\cn|$, we restrict our attention to the set  
 \[
V_{n, \sigma_1,\sigma_2} = \left\{k\in \R^3\setminus \{\vec 0,-n\}: |F|\leq \delta \quad\text{and}\quad\frac {|\cn|} 2 \leq |\ck| \leq |\cn|  \right\},
\]
with $\delta$  depending on $\max\{|\ck|,|\cm|,|\cn|\}=|\cn|$, hence  fixed.   Note that  $V_{n, \sigma_1, \sigma_2}\subset \R^3\setminus \{\vec 0, -n\}$, so that we are not a-priori restricted to integer points. Similarly to our convention for $F$, we will simply write $V$ for the remainder of this subsection.

We now present the following theorem, which reduces the lattice counting problem for the set $\cN_0$ to volume estimates for $V$.   This is possible as the monotonicity of the function $F$ in each Cartesian direction is well controlled in our case. The result holds under a fixed choice of $n$ and is valid up to a remainder term of order $|\cn|^2.$
\begin{theorem}\label{counting3}
Let $n\neq 0 $ and $\delta\in\left[0,\min\{\frac \eta 2, \frac 1 2\}\right)$ be fixed.  Then there exists a  constant $C>0$ independent of $n, \delta,\sigma_1,\sigma_2,\sL_1,\sL_2$ such that the cardinality of the set
\[
\cN_0:=\big\{(n,k,m)\in\cN: |\cn|\ge |\ck| \ge |\cm|\big\}
\] 
satisfies the following estimate
\[
 \sum_{{k \in \Z^3 }} \ind_{\cN_0}(n,k,-n-k)\le C|\cn|^2(\sL_1\sL_2+\sL_1+\sL_2)+ \sum_{(\sigma_1,\sigma_2)\in\{+,-\}^2}vol(V_{n, \sigma_1,\sigma_2}).
\]

\end{theorem}
\begin{proof}
 Let  $ k_H \in \Z^2$, be fixed. Then, we have:
\[
\dfrac{\partial F}{\partial k_3 }=(\eta^2-1)\left[ \sigma_2\dfrac{k_3  |\ck_H|^2}{ |\ck_\eta| |\ck|^3 } - \dfrac{m_3  |\cm_H|^2}{ |\cm_\eta| |\cm|^3 }\right].
\]
The critical points of $F$ for fixed $n$ and $\ck_H$ are included in the set of solutions of the following equation
\begin{align*}
k_3^2  |\ck_H|^4  \left[|\cm_H|^2 +\eta^2 (k_3 + n_3)^2 \right] &  \left[|\cm_H|^2 + (k_3 + n_3)^2  \right]^3\\
 & - (k_3+n_3)^2  |\cm_H|^4  { \left( |\ck_H|^2 + \eta^2 k_3^2 \right)} \left[|\ck_H|^2 + k_3 ^2  \right]^3= 0,
\end{align*}
{ after eliminating radicals in  $\dfrac{\partial F}{\partial k_3 } = 0$ and recalling that $\ck_\eta=(\ck_1, \ck_2, \eta k_3)$.}
The expression on the left-handside is a polynomial of degree 10 in $k_3$, with leading order coefficient $ \eta^2(|\ck_H|^4-|\cm_H|^4).$  Thus, for fixed $k_H \in \Z^2$, it follows that  $\R$ can be split in at most 11 intervals  where $F$ is strictly monotonic in  $k_3$. Note that  $F$ can be constant as a function of $k_3$ when $| \ck_H | = | \cm_H |$  and $n_3 = 0$. Nevertheless, 
we still have
\begin{equation}\label{k3sum}
\sum_{k_3 \in \Z} \ind_{V}(k_H, k_3) \leq 11 + \int_{\R} \ind_V (k_H, k_3) \rd k_3,
\end{equation}
for any fixed $ \ck_H $ and $ n \neq 0 $.
On the other hand, we claim that in the horizontal directions, $k_1, k_2,$ similar estimates hold true. Namely, we have
\begin{equation}\label{claimi}
\sum_{k_i \in \Z} \ind_{V}(k_H, k_3) \leq 11 + \int_{\R} \ind_V (k_H,  k_3) \rd k_i, \quad i \in \{ 1,2 \},
\end{equation}
when $n \neq 0$, $k_j$,  with $j \in \{1,2\}$, $i \neq j$, and $k_3$ are fixed.

 Indeed, critical points in the $k_i$-direction, for $i=1,2$, correspond to solutions of
 \[
  \sigma_2\dfrac{\ck_i  k_3^2}{ |\ck_\eta| |\ck|^3 } - \dfrac{\cm_i  m_3^2}{ |\cm_\eta| |\cm|^3 } = 0,
\]
 which are included in the roots of the polynomial
\begin{align}
\ck_i^2  k_3^4  |{\cm}_\eta|^2 &  \left[|\cm_H|^2 + (k_3 + n_3)^2  \right]^3 - \left( \ck_i + \cn_i \right)^2 \cm_3^4 |{ \ck}_\eta|^2 \left[|\ck_H|^2 + k_3 ^2  \right]^3.
\end{align}
 In the case $k_3^4 - m_3^4 \neq 0,$  the critical points in the $k_i$-direction are among the roots of a polynomial of degree 10 in $\ck_i$, which proves \eqref{claimi}  in that context.  Otherwise, $F$ is constant in $k_i$ and  \eqref{claimi} still holds.
Then, using  \eqref{claimi}, we  sum \eqref{k3sum} over $k_1\in\Z$ with $|\ck_1|\leq |\cn|$,   to deduce that:
\begin{align}\label{areaL1}
&\sum_{k_1,k_3 \in \Z^2} \ind_{V}(k_1,k_2,k_3) \leq C \sL_1 {|\cn|} + \int_{\R^2} \ind_V (k_H, k_3) \rd k_1 \rd k_3.
\end{align}
A further summation of \eqref{areaL1} over $k_2 \in \Z$ with $| \ck_2 |\leq |\cn|,$ combined with  \eqref{claimi}, up to an exchange between integrals and sums when necessary, then yields the result. 
\end{proof}
\end{subsection}
\begin{subsection}{Some properties of FFF   interactions under near resonance condition}
We begin this subsection with a  lemma on the possible choices of  signs for a near resonant FFF triplet. As it turns out, the  simple bound $\omega_k >\delta$, in the  rotating stratified Boussinesq context,  leads to  a simplification in comparison to the rotating  Navier-Stokes case in \cite{BCZNS}. 
\begin{lemma}\label{ordering}
 Consider $\sigma_1,\sigma_2,\sigma_3 \in \{\pm\}$ and 
 the  ordered dispersion relation values $\omega_1 \geq \omega_2 \geq \omega_3 > \delta\ge0$.  If the near resonance condition
\[
| \sigma_1 \omega_1+  \sigma_2\omega_2 +  \sigma_3 \omega_3 |\leq\delta
\]
holds true,   then we necessarily have 
\[
\sigma_1 \neq \sigma_2\quad\text{and}\quad \sigma_2= \sigma_3,
\]
 namely, $(\sigma_1,\sigma_2,\sigma_3)=(+,-,-)$ or $(-,+,+)$.
\end{lemma}
\begin{proof}
 It suffices to prove { the statement} with the additional assumption $\sigma_1=+$.

We will argue by contradiction. First, we suppose instead that $\sigma_1 =  \sigma_2$. Then, the near resonance condition implies that $\omega_1 + \omega_2 +   \sigma_3  \omega_3{\le\delta}$. Since  $ \omega_ 2 +   {\sigma_3}\omega_3\ge0$  by the ordering assumption, we infer that $ \omega_1\leq \delta$, leading to a contradiction.

 Similarly, we suppose instead $ \sigma_2  \neq \sigma_3$.
 Then the first part of the  conclusion guarantees that $\sigma_2 = -$,   hence  $\sigma_3 =+ $. Then  
$ \omega_1 - \omega_ 2 + { \omega_3\le \delta}$, by  the near resonance condition. Since $\omega_1 - \omega_ 2\ge0$  by the ordering assumption, we infer
 that $\omega_3  \leq \delta$,   leading to a contradiction.
\end{proof}
\begin{corollary}\label{orderingeta}
Let  $\delta\in [0,\min\{\frac{\eta} 2,\frac 1 2\})$ be fixed. Then a necessary condition for the set of FFF near resonances to be non-empty is
\[
{\begin{cases} \eta \geq 2-\delta,&\text{\quad if \; }\eta>1,\\
\eta \leq \frac{\delta+1} 2,&\text{\quad if \; }\eta<1.
\end{cases}}
\] 
\end{corollary}
\begin{proof}
Let $\omega_1\geq \omega_2 \geq \omega_3$ be the ordered eigenvalues, in decreasing order, from Lemma \ref{ordering}.
We have: $-\delta\le \omega_1-\omega_2-\omega_3$, due to near resonance and Lemma \ref{ordering}. Since $\omega_1,\omega_2,\omega_3 \in  \left[\min\{1,\eta\},\max\{1,\eta\}\right]$,  the claim follows.
\end{proof}
\begin{corollary}\label{orderingsmaller}
 Let $\delta \in { [}0,\min \{\frac{\eta} 2, \frac 1 2\})$ and $\eta\neq 1$ be fixed. Then, under the assumptions of  Lemma \ref{ordering}, we have
\[
\omega_3 \leq\frac{\max\{\eta, 1\}+\delta} 2.
\]
\end{corollary}
\begin{proof}
  The near resonance condition and { Lemma \ref{ordering}} directly imply that
$-\delta\le\omega_1-2\omega_3$, with the claim following from the trivial upper bound on $\omega_1$.
\end{proof}
\end{subsection}
\begin{subsection}{Elliptic integrals}
In order to calculate the volume of the set $V_{n, \sigma_1, \sigma_2}$, we follow the strategy of \cite[Section 5]{BCZNS} which leads to the study of certain elliptic integrals. First, we recall that $ n\in\Z^3\setminus\{\vec 0\}$ and $\delta>0$ are fixed.  Since the volume element of the integral under consideration scales like $\sL_1\sL_2\rd \ck = \rd k$, it suffices to compute the volume for the case $\sL_1=\sL_2=1.$

 Next, the  volume integral 
\[
\int_{\R^3} \ind_{\vn}(k) \rd k
\]
 is expressed in spherical coordinates   $\left(\lambda_k,  \theta_k, \phi_k \right)  \in [\frac 1 2, 1] \times [0, \pi] \times  [0, 2\pi]$, with a rescaled radial component $\lambda_k$, similarly to the proof of Lemma \ref{mixedvolumel}, but with $M=|n|$. In particular, $\lambda_k:=\frac{|k|}{|n|}$, thus 
\[
 k =|n|\lambda_k (\sin\theta_k \cos \phi_k, \sin \theta_k \sin \phi_k, c_k),\quad\text{for}\quad k\in\Z^3.
\]
First, a change of variable from the azimuthal angle to $c_m$ is performed. We define the set
\[
A^F := \{(c_k,c_m)\in[-1,1]^2:\min_{(\sigma_1,\sigma_2)\in\{\pm\}^2}| \sigma_1{ \omega_{n}}+\sigma_2 { \omega_k}+\omega_m|\leq \delta \}.
\]
 An initial estimate 
\[
vol(V_{n, \sigma_1, \sigma_2}  ) \leq 16  |n|^3 \int_0^\pi    \int_{\frac 1 2}^1 \int_{-1}^1  \ind_{A^F} \mq^*(\lambda_k,\phi_k,\theta_k) |c_m| \sin\theta_k  \rd c_m  \rd \lambda_k \rd \theta_k
\]
can be then obtained, after examining the constraints on the magnitude of the wavevectors. In that context,  the expression $\mq^* |c_m|$, which depends on all three variables $(\lambda_k, \phi_k,\theta_k)$,  is a result of the non-zero Jacobian after our change of variable.
Then, we define
\[
\mQ:= |c_m|\int_{\frac 1 2}^1 \mq^* \rd\lambda_k,
\]
so that it suffices to estimate an expression of the form
\[
16  |n|^3  \int_{-1}^1  \int_0^\pi   \ind_{A^F}   \mQ(\theta_k, c_m) \sin\theta_k  \rd \theta_k \rd c_m,
\]
due to  the Fubini-Tonelli theorem.
The behaviour of $\mQ$ can be further quantified using the theory of elliptic integrals. In particular, we  recall  \cite[Lemma 5.4]{BCZNS}.
\begin{lemma}\label{decreasingL}
 Let $c_m,c_k,c_n \in (-1,1)\setminus\{0\}$ be pairwise distinct. Moreover, let $\varsigma_i$, $i=1,\ldots,4$, denote the elements of the set $\{1,|c_k|,|c_m|,|c_n|\}$ arranged in decreasing order.  Then the following statement holds true:
\begin{equation}\label{Q1}
\mQ \leq \frac C {\sqrt{(1 -\varsigma_3) (\varsigma_2 - \varsigma_4)}}  \left[{1+\log\sqrt{\frac {(1 -\varsigma_3)} {(1 - \varsigma_4)}\frac {(\varsigma_2 -\varsigma_4)} {(\varsigma_2 - \varsigma_3)}}}\right],
\end{equation}
for a constant $C>0$.
In addition, we have the following estimate:
\begin{equation}\label{Q3}
\mQ \lesssim \frac 1 {\sqrt{(1-\varsigma_3)(\varsigma_2 -\varsigma_4)}},\quad\text{when}\quad |c_n|\leq \min \{|c_m|,|c_n|\}.
\end{equation}
\end{lemma}
 \begin{remark}
Combining \eqref{Q1} with the inequality $\log x \leq x-1$,
we obtain the following estimate
\begin{equation}\label{Q2}
\mQ \lesssim \dfrac 1 {\sqrt{1 - \varsigma_4}\sqrt{\varsigma_2 - \varsigma_3}} \lesssim_{\eta} \dfrac 1 {\sqrt{\varsigma_2 - \varsigma_3}},
\end{equation}
with the second inequality holding true in view of { Corollary} \ref{orderingsmaller}.
\end{remark}
A technical necessity  that stems from the previous Lemma is the consideration of the possible orderings of $|c_k|, |c_m|, |c_n|.$   However, the number of cases we need to take into account is limited, as  the sign of the term corresponding to the wavevector labelled with $m$ is always positive. In particular,  Lemma \ref{ordering} yields 6 possible cases.

\end{subsection}
\begin{subsection}{The volume estimate for FFF  interactions}
In view of Lemmas \ref{ordering} and \ref{decreasingL}, we can carry out a more refined estimate compared to that of \cite{BCZNS}.   In more detail,  $\omega_k,\omega_m,\omega_n$ are strictly bounded away from zero, a property which was not available in \cite{BCZNS}, while  their ordering with respect to $c_k,c_m,c_n$ is either fully preserved or reversed.
\begin{theorem}\label{volumeT}
Let $\delta \in \left(0, \min \{\frac{\eta} 2,\frac 1 2\}\right)$, $\eta\neq 1$,  $\sigma_1, \sigma_2 \in \{\pm \}$ and $n \in \Z^3 \setminus  \{\vec 0\}$  be fixed. Moreover, consider the set
\[
V_{n, \sigma_1, \sigma_2} : = \left\{k \in \R^3 \setminus \{\vec 0,n\}: |F_{n, \sigma_1, \sigma_2}(k)|\leq \delta \quad\text{and}\quad \frac 1 2 |\cn| \leq |\ck| \leq |\cn| \right\}. 
\]
Then, the following estimate holds true
{\begin{equation}\label{volumeE}
vol(V_{n, \sigma_1,\sigma_2})  \leq C(\eta)\delta \sL_1 \sL_2 |\cn|^3,
\end{equation}
for a positive  constant $C$} independent of $\delta, \sL_1, \sL_2, n, \sigma_1, \sigma_2.$
\end{theorem}

\begin{proof}
First, we recall the fact that $\omega_k= f( |c_k|)$  for an either strictly increasing function or strictly decreasing function $f_\eta$, namely $f_\eta (x):=\sqrt{1 + (\eta^2 - 1) x^2}$.  The monotonicity of $f_\eta$ depends on the region of $\eta$ under consideration, when $x>0$. We initially focus on the case $1<\eta$, where $f_\eta$ is strictly increasing.  Then, a reversal in monotonicity for the remaining $\eta'$s will still yield the result, up to some necessary permutations of the wavenumbers involved.

We introduce the following sets:
\[\begin{aligned}
A_0&:=\big\{(c_k,c_m)\in[-1,1]^2 : c_kc_m(1-c_k^2)(1-c_m^2)(c_m^2- c_k^2)(c_k^2-c_n^2)(c_n^2-c_m^2)=0\big\},\\
A_1&:=\big\{(c_k,c_m)\in[-1,1]^2 : |c_k|<\min\{|c_n|,|c_m|\}\big\}\setminus A_0,\\
A_2&:=\big\{(c_k,c_m)\in[-1,1]^2:|c_m|<\min\{|c_n|,|c_k|\}\big\}\setminus A_0, \\
A_3&:=\big\{(c_k,c_m)\in[-1,1]^2:|c_n|<\min\{|c_k|,|c_m|\}\big\}\setminus A_0,\\
B_1&:=\big\{(c_k, c_m)\in[-1,1]^2:|c_n|<|c_m| \big\},\\
B_2&:=\big\{(c_k,c_m)\in[-1,1]^2:|c_k|<|c_m| \big\}.
\end{aligned}\]
It suffices to prove \eqref{volumeE} for $n$ in a dense subset of $\R^3,$ so that we can exclude some problematic values.
In particular, we fix $n\in \Z^{3}\setminus \{\vec 0\}$, with $n_1 n_2 n_3 \neq 0$. Then, for $i=1,2,3$, we define
\[A_i^F := A_i  \cap A^F.
\]
We will extensively use the  following  relations in what follows:
\begin{equation}\label{changec}
 d c_k = \pm\frac 1 {\sqrt {|\eta^2 -1| }} \frac {\omega_k} {\sqrt{|\omega_k^2 - 1|}} d \omega_k\quad\text{and}\quad c_k^2 - c_m^2 = \frac {\omega_k^2 - \omega_m^2} {\eta^2 - 1},
\end{equation}
which hold true for any combination of $c_k,c_n,c_m$. Finally, we recall that due to Corollary \ref{orderingeta}  we need to examine  the  
rotation dominated regime, $\eta \geq 2-\delta$, and the stratification dominated regime, $\eta \leq \frac{\delta+1} 2,$   separately.

\underline{\textbf{The rotation dominated regime}}\\
 When $\ind_{A_1} = 1$,  the estimate \eqref{Q2} is applicable. Nevertheless, we need to take into account whether $|c_n|< |c_m|$ or not. In the former case, the only possible choice of signs for $(\omega_n, \omega_k)$ is $(-, -)$, due to Lemma \ref{ordering}.  Then, the following relations hold true
\begin{equation}\label{divisorlimits}
 \omega_m- \omega_n\in (\omega_k - \delta, \omega_k + \delta)\quad\text{and}\quad
\sqrt 2 \sqrt{( \omega_k -\delta)} < \sqrt{\omega_m^2 - \omega_n^2 },
\end{equation}
 as a consequence of the near resonance condition and our ordering considerations. As all the estimates depend on the absolute value of the quantities under consideration, we restrict our attention to the case  $c_k, c_m>0.$
Then we use \eqref{Q2},  substituting $c_m$ with $\omega_m$ in the volume integral, in conjunction with \eqref{changec} and \eqref{divisorlimits}, so that
\begin{align*}
&\int_0^1 \int_{|c_n|}^{1} \dfrac{\ind_{A_1^F\cap B_1}}{\sqrt{c_m^2 -c_n^2}} dc_m dc_k  = \int_0^{|c_n|} \int_{\omega_k + \omega_n - \delta }^{\omega_k + \omega_n + \delta}\dfrac{1}{\sqrt{\omega_m^2 -\omega_n^2}}\frac {\omega_m} {\sqrt{\omega_m^2 - 1}} \rd \omega_m \rd c_k \\
& <  \sqrt2 \delta \int_{0}^{|c_n|} \frac 1 { \sqrt{( \omega_k -\delta)}} \frac {\omega_k+ \omega_n - \delta} {\sqrt{(\omega_k + \omega_n - \delta)^2 - 1}}dc_k.
\end{align*}
Since $\delta \leq \frac 1 2$ and $1\leq \omega_k, \omega_n$, the last integrand is uniformly bounded, thus proving our claim in that context.

The case where $(\omega_m,\omega_n, \omega_k)$ have sign $(+,-,+)$ requires a more delicate treatment, under the restriction that $\omega_m < \omega_n.$ First, we have the analog of \eqref{divisorlimits}
\begin{equation}\label{divisorlimits1}
   \omega_n - \omega_m \in(\omega_k - \delta,  \omega_k + \delta)\quad\text{and}\quad
\sqrt 2 \sqrt{( \omega_k -\delta)} < \sqrt{\omega_n^2 - \omega_m^2 }.
\end{equation}
Moreover, the near resonance condition together with \eqref{divisorlimits1}, the ordering imposed on the magnitude of the wavevectors, and the fact that $\delta< \frac 1 2$ imply that:
\begin{equation*}
\frac {\sqrt 5}  {2 \sqrt{\eta^2 -1}}<\frac {\sqrt {(2-\delta)^2 - 1} }{\sqrt{\eta^2 -1}}\leq |c_n|\leq \frac {|k_3|}{|\cn|} +\frac {|m_3|}{|\cn|}\leq |c_k| +2 |c_m|\leq 3|c_m|,
\end{equation*}
i.e.
$\omega_{m}>\frac{\sqrt{41}} 6.$ We recall that \eqref{Q2} is applicable.
Then, our last estimate on $\omega_m,$ combined with the fact that $\omega_k - \delta>\frac 1 2$, a change of variable from $c_m$ to $\omega_m$, \eqref{changec} and  \eqref{divisorlimits1}, allow us to estimate:
\begin{align*}
&\int_0^1 \int_{-1}^1 \dfrac{\ind_{A_1^F\setminus B_1}}{\sqrt{c_n^2 -c_m^2}}dc_m dc_k  <  \int_0^{|c_n|} \int_{\omega_n - \omega_k - \delta }^{\omega_n - \omega_k + \delta}\dfrac{1}{\sqrt{\omega_n^2 -\omega_m^2}}\frac {\omega_m} {\sqrt{\omega_m^2 - 1}} \rd \omega_m \rd c_k \\
& < C \delta \int_0^{|c_n|}  \frac 1 { \sqrt{( \omega_k -\delta)}} \rd c_k < C\delta,
\end{align*}
for an absolute constant $C>0$. 

The case $(c_k,c_m)\in {A_2^F}$  can be handled analogously, by permuting the roles of the three wavevectors where necessary. 

As far as the case $\ind_{A_3}=1$ is concerned, we further distinguish cases. If $|c_k|<|c_m|$, then the signs for $(\omega_m,\omega_n, \omega_k)$ are  $(+,-,-)$. Thus,  we have
\begin{equation}\label{divisorlimits2}
 \omega_m - \omega_n \in(\omega_k - \delta,  \omega_k + \delta). 
\end{equation}
We use \eqref{Q3}, change variables from $c_m$ to $\omega_m$ and recall  \eqref{changec},\eqref{divisorlimits2}, in order to deduce that
\begin{align*}
&\int_{0}^1 \int_{-1}^1 \dfrac{\ind_{A_3^F \cap B_2 }}{\sqrt{c_m^2 -c_n^2}}\dfrac{1}{\sqrt{1 -c_k^2}}dc_m dc_k  =  \int_{|c_n|}^{1} \int_{\omega_k + \omega_n - \delta }^{\omega_k + \omega_n + \delta}\dfrac{1}{\sqrt{\omega_m^2 -\omega_n^2}}\frac {\omega_m} {\sqrt{\omega_m^2 - 1}}\frac {1} {\sqrt{1 - c_k^2 }} d \omega_m dc_k \\
& < \sqrt2 \delta \int_{|c_n|}^{1} \frac 1 { \sqrt{( \omega_k -\delta)}} \frac {\omega_k+ \omega_n - \delta} {\sqrt{(\omega_k + \omega_n - \delta)^2 - 1}} \frac {1} {\sqrt{1 - c_k^2 }}\rd c_k\\
&\leq \sqrt2 C\delta,
\end{align*}
for a constant $C>0$ independent of $\delta,n,\eta$. Finally, when $|c_m|<|c_k|$ our last estimate still goes through, up to reversing the roles of $k$ and  $m$.\\
\underline{\textbf{The stratification dominated regime}}\\
The case $2\delta  < \eta<1$ presents a reversal in monotonicity.

 In more detail, when $\ind_{A_1} = 1$ and $|c_n|<|c_m|,$  then Lemma \ref{ordering} implies that  $(\omega_n, \omega_k)$ necessarily have  opposite signs and
\begin{equation}\label{divisorlimits3}
\omega_k-\omega_n\in(\omega_m-\delta,\omega_m+\delta).
\end{equation}
We will use \eqref{Q1} and change coordinates first from $(c_k,c_m)$ to $(\omega_k,\omega_m)$ and then  from $(\omega_k, \omega_m)$ to $(\omega_k,-\omega_n+\omega_k+\delta'),$
so that $\delta'\in(-\delta,\delta)$ by \eqref{divisorlimits3}.
We also have
\begin{itemize}
\item $\omega_k^2 -\omega_m^2<\omega_k^2 -\eta^2,$ using the trivial lower bound for $\omega_m$,
\item $\omega_k^2 -\omega_m^2 >(\omega_n -\delta)(\omega_k+\omega_m)>\eta^2 ,$ by \eqref{divisorlimits3} and the fact that $\delta<\frac \eta 2$,
\item $\frac {\omega_m} {\sqrt{1 - \omega_m^2}}\leq C$  by Corollary \ref{orderingsmaller},
\item $\omega_k\in(\eta+\omega_n+\delta',\min\{1,2\omega_n+\delta'\}):=(b,a)$, due to the fact that $\omega_m<\omega_n.$
\end{itemize}
Thus, via \eqref{Q1}, \eqref{changec} and our change of variables
\begin{align*}
&\int_0^1 \int_0^1 \frac {\ind_{A_1^F\cap B_1} \left[{1+\log\sqrt{\frac {(1 -c_n^2)} {(1 - c_k^2)}\frac {(c_m^2 -c_k^2)} {(c_m^2- c_n^2)}}}\right]} {\sqrt{(1 -c_n^2) (c_m^2 - c_k^2)}} \rd c_k \rd c_m\\
&=\int_0^1 \int_0^1 \frac {\ind_{A_1^F\cap B_1}\left[{1+\log\sqrt{\frac {(\omega_n^2-\eta^2)} {(\omega_k^2 -\eta^2)}\frac {(\omega_k^2 -\omega_m^2)} {(\omega_n^2- \omega_m^2)}}}\right]} {\sqrt{(\omega_n^2-\eta^2) (\omega_k^2 - \omega_m^2)}}  \rd c_k \rd c_m\\
&=\frac{1}{1-\eta^2}\int_{\eta}^{\omega_n} \int_{\eta}^1 \frac {\ind_{A_1^F} \left[{1+\log\sqrt{\frac {(\omega_n^2-\eta^2)} {(\omega_k^2 -\eta^2)}\frac {(\omega_k^2 -\omega_m^2)} {(\omega_n^2- \omega_m^2)}}}\right]} {\sqrt{(\omega_n^2-\eta^2) (\omega_k^2 - \omega_m^2)}} \frac {\omega_k} {\sqrt{1-\omega_k^2}}\frac {\omega_m} {\sqrt{1-\omega_m^2}} \rd \omega_k \rd \omega_m\\
&\leq C\int_{\eta}^{\omega_n} \int_{\eta}^1  \frac {\ind_{A_1^F}\left[{C_1+\log\sqrt{\frac {(\omega_n-\eta)} {(\omega_n- \omega_m)}}}\right]} {\sqrt{(\omega_n-\eta) }}  \frac {1} {\sqrt{1-\omega_k}}\rd \omega_k \rd \omega_m\\
&=C\int_{-\delta}^{\delta} \int_{b}^a  \frac {\ind_{A_1^F}\left[{C_1+\log\sqrt{\frac {(\omega_n-\eta)} {(2\omega_n- \omega_k+\delta')}}}\right]} {\sqrt{(\omega_n-\eta) }}  \frac {1} {\sqrt{1-\omega_k}}\rd \omega_k \rd \delta'\\
&\leq C\int_{-\delta}^{\delta} \int_{b}^a  \frac {\left[{C_1+\log\sqrt{\frac {(\omega_n-\eta)} {(a- \omega_k)}}}\right]} {\sqrt{(\omega_n-\eta)(a-\omega_k) }}  \rd \omega_k \rd \delta'\\
&\leq C\int_{-\delta}^{\delta} \int_0^{\sqrt{\frac{a-b}{\omega_n-\eta}}}   \left[{C_1-\log y}\right]\rd y \rd \delta'\leq C \delta,
\end{align*}
where we further changed variables from $\sqrt{\frac {(a- \omega_k)}{(\omega_n-\eta)} }$ to $y$.  Note that the inner integral has upper limit not exceeding 1.

When $\ind_{A_1^F}=1$ and $|c_m|<|c_n|,$ we follow a similar procedure to the above. In particular, we take into account \eqref{divisorlimits3} and change coordinates from $(c_k,c_m)$ to $(\omega_k,\omega_m)$, followed by a change from$(\omega_k,\omega_m)$ to $(\omega_k, -\omega_n+\omega_k+\delta').$ Then, \eqref{divisorlimits3} implies that 
\begin{itemize}
\item $\omega_m <1-\frac \eta 2,$ 
\item $\omega_m-\omega_n\in(\max\{0,\omega_k-2\omega_n-\delta\},\omega_k-2\omega_n+\delta)$,
\end{itemize}
so that we can estimate
\begin{align*}
&\int_0^1 \int_{0}^{1} \dfrac{\ind_{A_1^F\setminus B_1}}{\sqrt{c_n^2 -c_m^2}} dc_k dc_m  =\frac{1}{\sqrt{1-\eta^2}}  \int_{\omega_n }^{\omega_k} \int_{\omega_m}^{1} \dfrac{\ind_{A_1^F}}{\sqrt{\omega_m^2 -\omega_n^2}}\frac {\omega_m} {\sqrt{1-\omega_m^2 }} \frac {\omega_k} {\sqrt{1-\omega_k^2 }} \rd \omega_k \rd \omega_m \\
& <  \frac{C}{\sqrt{1-\eta^2}} \int_{\omega_n }^{\omega_k}\int_{\omega_m}^{1} \dfrac{\ind_{A_1^F}}{\sqrt{\omega_m -\omega_n}} \frac {1} {\sqrt{1-\omega_k }} \rd \omega_k \rd \omega_m\\
&= \frac{C}{\sqrt{1-\eta^2}}  \int_{-\delta}^{\delta} \int_{2\omega_n - \delta'}^{1}\dfrac{1}{\sqrt{\omega_k - 2 \omega_n +\delta'}} \frac {1} {\sqrt{1-\omega_k }} \rd \omega_k \rd \delta'\\
&= \frac{C}{\sqrt{1-\eta^2}}  \int_{-\delta}^{\delta} \int_0^1\dfrac{1}{\sqrt{1-y}} \frac {1} {\sqrt{y }} \rd y\rd \delta'<C\delta,
\end{align*}
where we further changed variables as $y=\frac {1-\omega_k} {1-2\omega_n + \delta'}$.

{ In the case that $\ind_{A_2^F}=1,$ it suffices to repeat the preceeding steps, up to a permutation of the wavevectors.}

Finally, when $\ind_{A_3}^F=1,$ we give the details on the estimate for the case $|c_k|<|c_m|,$ i.e. when $\omega_m<\omega_k<\omega_n$.  Hence, Corollary \ref{orderingsmaller} implies that we have $\omega_m \leq\frac 3 4$. In addition, it holds that 
\begin{equation}\label{divisorlimits5}
\omega_m \in(\omega_n-\omega_k-\delta, \omega_n-\omega_k+\delta),
\end{equation}
 due to the near resonance condition. We change variables from $c_m$ to $\omega_m$, using \eqref{Q3}, \eqref{divisorlimits5}, the aforementioned boundedness of $\omega_m$ away from 1 and the fact that $2\eta<\omega_m+\omega_n$, in order to deduce:
\begin{align*}
&\int_{-1}^1 \int_{-1}^1 \dfrac{\ind_{A_3^F \cap {B_2}}}{\sqrt{c_m^2 -c_n^2}} \dfrac 1 {\sqrt{1-c_k^2}}\rd c_m \rd c_k  \\
& <  \int_{|c_n|}^1 \int_{\omega_n-\omega_k -\delta }^{\omega_n-\omega_k +\delta } \dfrac{\ind_{A_3}}{\sqrt{\omega_n^2 -\omega_m^2}}\frac {\omega_m} {\sqrt{1 - \omega_m^2 }}\dfrac 1 {\sqrt{1-c_k^2}}  \rd\omega_m \rd c_k \\
&< \frac {C }{\sqrt\eta}   \int_{|c_n|}^1 \int_{\omega_n-\omega_k -\delta }^{\omega_n-\omega_k +\delta }{\frac {1} {\sqrt{\omega_n -\omega_m }}}\dfrac 1 {\sqrt{1-c_k^2}}  \rd\omega_m \rd c_k \\
&< \frac {C \delta}{\sqrt\eta}   \int_{|c_n|}^1 \frac {1} {\sqrt{\omega_k -\delta }}\dfrac 1 {\sqrt{1-c_k^2}}   \rd c_k < \frac {C \delta}{\eta}   \int_{|c_n|}^1 \dfrac 1 {\sqrt{1-c_k^2}}   \rd c_k<C\delta. 
\end{align*}
\end{proof}
\end{subsection}
\end{section}
\begin{section}{Proof of the main results}\label{proofS}
\begin{subsection}{Proof of Theorem \ref{restrictedfT}(FFF estimates)}
 We prove Theorem \ref{restrictedfT} using a similar strategy to \cite[Theorem 1.3]{BCZNS} in the rotating Navier-Stokes context. 
\begin{proof}[Proof of Theorem \ref{restrictedfT}]
First, we recall the sign convention  for the bilinearity from \eqref{bcoeff} and the conjugation property of Remark \ref{urealR}. Then, we use  Parseval's identity and the symmetry between $m,n$ in the resulting sum, due to incompressibility, in order to get:
\begin{align*}
&2\ip{\pD^{\ell} \wt B_f (\umm_f,\vmm_f),\pD^{\ell} \vmm_f} \\
 &= \ri |\T| \skm \sum_{\sigma_1\sigma_2\sigma_3\neq0} \left( |\cn|^{2\ell} - |\cm|^{2\ell} \right)\cip{r_k^{\sigma_1}}{\cm'} \cip{r_m^{\sigma_2}}{{r_{n}^{\sigma_3}}} u_k^{\sigma_1}v_m^{\sigma_2}  v_{n}^{\sigma_3} \ind_{\cN^{\textnormal {FFF}}} (k,m,n).
\end{align*}
In the range of $\ell$ under consideration, we use   the triangle inequality and the mean value theorem to obtain $\left| |\cn|^\ell-|\cm|^\ell\right| \leq \ell |\ck| \max\{|\cn|^{\ell-1}, |\cm|^{\ell-1} \}$.  In addition,  $|\cip{r_k^{\sigma_1}}{\cm'}| \leq \min \{|\cm|,|\cn|\}$ holds true, due to incompressibility. Then,  we complete the proof by combining Lemma \ref{restricted:L} for $\beta=2$  and our counting results from Theorems \ref{counting3}, \ref{volumeT} together with the choice of $\delta$ specified in Theorem \ref{restrictedfT}.
\end{proof}
\end{subsection}
\begin{subsection}{Energy estimates in $L^2$}
As a starting step towards the proof of Theorem \ref{globalexistenceT}, we show that the standard $L^2$ energy inequality holds true for the restricted system \eqref{restrictedS}.
\begin{lemma}\label{L2le}
Let $\Ummt_0\in H^{\ell}(\T;\R^4)$  with $\ell\ge1$ be a divergence free and zero-mean vector field.
 If  $\Ummt$ is a solution of \eqref{restrictedS} with initial data $\Ummt_0$   for $t\in [0,T),$ then: 
\begin{equation}\label{L2}
\| \Ummt(T) \|_{L^2}^2 + 2 \nu_{min}  \int_0^T \|\Ummt \|_{H^1}^2 \rd t  \le \| \Ummt_0\|_{L^2}^2.
\end{equation}
\end{lemma}
\begin{proof}
We test \eqref{restrictedS} with $\Ummt$, so that
\[
  \partial_t \|\Ummt \|_{L^2}^2 +2\ip{\wt A \Ummt, \Ummt}  + 2 \ip{\wt B(\Ummt,\Ummt),\Ummt} = 0
\] 
follows from the skew-symmetry of $\Le$.  Moreover,
 Lemma \ref{orthogonalityL} and Lemma \ref{ellipticity} imply that
$ {\nu_{min}}   \|\Ummt \|_{H^1}^2 \leq  \ip{\wt A \Ummt, \Ummt}.$
We  also have $\ip{\wt B (\Ummt,\Ummt),\Ummt} = 0$, due to Corollary \ref{L2l0},  with \eqref{L2} following after an integration in time.
\end{proof}
\end{subsection}
\begin{subsection}{Proof of Theorem \ref{restrictedT} and slow output estimates}\label{slowP}
We now present the proof of the FFS estimates of Theorem \ref{restrictedT}. Its main ingredients are Lemma \ref{restrictedl}, which is applicable due to Corollary \ref{mixedvolumec},  and a growth bound on the interaction coefficients for the mixed terms, which is based on the analysis of Appendix \ref{appendixs}.
\begin{proof}[Proof of Theorem \ref{restrictedT} ]
We recall   \eqref{bcoeff} and  Remark \ref{urealR}. Then, due to  the $k,m$ symmetry of $\ind_{\cN_{FFS}}(k,m,n)$,  the mixed part of our slow approximation is given by
\begin{align*}\label{slowcoef}
\wt B_s (\umm_f,\umm_f)& = \sum_{k,n,m;conv} \left( B_{kmn}^{+ - 0}(\umm,\umm) +  B_{kmn}^{- + 0}(\umm,\umm) \right)\ind_{\cN^{\textnormal {FFS}}}{(k,m,n)}\\
&= \ri|\T| \skm   S_{kmn}^{+-0} u_k^+ u^- _m r_n^0  \ind_{\cN^{\textnormal {FFS}}} (k,m,n),
\end{align*}
with the interaction coefficients defined via
\begin{equation}\label{slowcoeff}
S_{kmn}^{+-0} =  \left( \rfp\cdot \cm'\right)  \cip{\rfmm}{r_n^0 } +  \cip{\rfmm}{\ck' } \cip{\rfp}{r_n^0 }.
\end{equation}
Using Parseval's identity, we have
\[
\ip{\pD^{\ell} \wt B_s (\umm_f,\umm_f),\pD^{\ell}\wmm}  = \ri|\T| \skm |\cn|^{2 \ell}  S_{kmn}^{+-0} u_k^+ u_m^-  w_n^0  \ind_{\cN^{\textnormal {FFS}}}{(k,m,n)}.
 \]Moreover, we have  
\[
|S_{kmn}^{+-0}| \lesssim_{\eta} \left | (\omega_k - \omega_m) \right|  |\ck| |\cm||\cn|^{-1}
\]
via  Lemma \ref{scof}, {and $|\cn|^\ell \lesssim_\ell |\ck|^\ell+ |\cm|^\ell$ via the triangle inequality, when $\ell>0$.} Then we combine the preceding estimates, \eqref{bandwidthffs}, and Lemma \ref{restrictedl} for the resulting convolution sum, with $\mu=3 - \frac{\xi} 2$,  in order to obtain
\begin{align*}
&\left| \ip{\pD^{\ell}\wt B_s (\umm_f,\umm_f), \pD^{\ell}\wmm}  \right|  \lesssim_{\eta,\T} \skm\left | (\omega_k - \omega_m) \right|   |\ck| |\cm||\cn|^{2\ell-1}  |u_k^{+}| |u_m^{-}|  |w_{n}^0| \ind_{\cN^{\textnormal {FFS}}}{ (k,m,n)}\\
&\lesssim_{\eta, {\ell}}  \skm \delta^* (k,m,n)   \left(|\ck|^{\ell_1} + |\cm|^{\ell_1}\right) |\ck| |\cm|  |\cn|^{\ell_2}  |u_k^+| |u_m^-|  |w_{n}^0| \ind_{\cN^{\textnormal {FFS}}}(k,m,n)\\
&\lesssim_{\eta, {C_{\delta*}}}  \skm    \left(|\ck|^{\ell_1} + |\cm|^{\ell_1}\right)  |\ck|^{1-\frac \xi 2} |\cm|^{1-\frac \xi 2}  |\cn|^{\ell_2}  |u_k^+| |u_m^-|  |w_{n}^0| \ind_{\cN^{\textnormal {FFS}}} (k,m,n)\\
&\lesssim \|\pD^{\frac 5 2 - \frac {3\xi} 4 - a  +\ell_1}  \umm_f \|_{L^2} \|\pD^{1 + a-\frac \xi 2}  \umm_f \|_{L^2}    \|\pD^{\ell_2} \wmm_s \|_{L^2}.
\end{align*}
\end{proof}
\begin{proof}[Proof of  Corollary \ref{restrictedlpvC}]
Using Parseval's identity, and Lemma \ref{orthogonalityL} we have
\[
\ip{\lpv \wt B_s (\umm_f,\umm_f),\lpv \wmm}  = \ri|\T| \skm |\cn_\eta|^{2}  S_{kmn}^{+-0} u_k^+ u_m^-  w_n^0  \ind_{\cN^{\textnormal {FFS}}}{ (k,m,n)}.
 \]
 The result then follows from  adapting the proof of Theorem \ref{restrictedT}, by using a Fourier multiplier with symbol $|\cn_\eta|$ instead of $\pD$ and setting $\ell_1=0$ and $\ell_2=1$.
\end{proof}
We proceed by examining the  SSS term $\wt B_s ( \Ummt_s, \Ummt_s).$
In that direction, we remark that the following norm equivalence
\begin{equation}\label{normequivalence}
\min \{\eta^{-1}, 1\} \|\wt Q \|_{H^{\ell-1}} \leq \| \Ummt_s  \|_{H^{\ell}} \leq \max \{\eta^{-1}, 1\} \|\wt Q \|_{H^{\ell-1}}
\end{equation}
holds true due to  \eqref{eqlpvslow},  where $\wt Q=\lpv \Ummt$. Thus, in view of Lemma \ref{orthogonalityL}, and the commutability properties of
$\wt\nu_{11},\wt{\nu}_{22}$ in Section \ref{restrictedLB}, it suffices to derive estimates for the following system
\begin{align}\label{approxeqlpv}
& \partial_t  \wt Q+  \lpv \wt B_s  (\Ummt, \Ummt)- \wt \nu_{11}  \Delta \wt Q= 0 \\
& \partial_t  \Ummt_f + \wt B_f(\Ummt, \Ummt) - \wt \nu_{22} \Delta \Ummt_f = N \Le \Ummt_f \label{approxeqlpv1}
\end{align}
instead of \eqref{restrictedS}.
An advantage of working with the linear potential vorticity equation is that the SSS term,  even though present in the PDE \eqref{approxeqlpv}, plays no part  on the $L^2$ energy  level.
In more detail, we have the following result from \cite{BMN2}.
\begin{lemma}\label{slowtransport}
The SSS transport term $B_s(\Ummt_s,\Ummt_s)$ satisfies the following identities
\begin{equation}\label{slowtransportgain}
\lpv B_s(\Ummt_s,\Ummt_s) = - \left( \dsn{-1} \nabla_H^{\perp} \wt Q \right)\cdot \nabla_H \wt Q =-\nabla_H \cdot \left(  \wt Q \dsn{-1}\nabla_H^{\perp} \wt Q  		 \right).
\end{equation}
In particular,
\begin{equation}\label{slowtransportgainL2}
\ip{\lpv B_s(\Ummt_s,\Ummt_s),\wt  Q} = 0.
\end{equation}
In addition, the following estimate holds true
\begin{equation}\label{lpvslowgain1}
\left|\ip{\lpv B_s(\Ummt_s,\Ummt_s) , \ds{\ell} \wt Q }
\right| \lesssim  \|\wt Q \|_{H^{\frac 1 2}}\|\wt  Q \|_{H^{\ell}}\|\wt  Q \|_{H^{\ell+1}},
\end{equation}
for all $\ell\in \R^+$,  with the implied constant independent of $\Ummt.$
\end{lemma}
\end{subsection}
\begin{subsection}{Proof of Theorem \ref{restrictedffT} and fast output estimates}
In this section we prove the convolution sum estimate of Theorem \ref{restrictedffT}, which follows a similar procedure to the proof of Theorem \ref{restrictedT}. In particular, we use Lemma \ref{restrictedl} once more, since Corollary \ref{mixedvolumec} holds true for the FSF interactions, but without any extra help from the interaction coefficients involved.
\begin{proof}[Proof of Theorem \ref{restrictedffT}]
 Using  Remark \ref{urealR} and Parseval's identity, we derive
\begin{align*}
\ip{\pD^{\ell} \wt B_f (\umm_f,\vmm_s),\pD^{\ell} \umm}  &= \ri|\T| \skm \sum_{\sigma_1 \sigma_3 <0} |\cn|^{2\ell}  S_{kmn}^{\sigma_1 0  \sigma_3 } u_k^{\sigma_1}v_m^0  u_{n}^{\sigma_3}    \ind_{\cN^{\textnormal {FSF}}}{ (k,m,n)},
\end{align*}
with $S_{kmn}^{\sigma_1 0  \sigma_3 } = \cip{r_k^{\sigma_1}}{\cm'} \cip{r_m^0}{{r_{n}^{\sigma_3}}}.$  We note that $|S_{kmn}^{\sigma_1 0  \sigma_3 }| \leq |\cm|$, for all possible choices of sign. Then,  using the fact that $|\cn|^{\ell_1} \lesssim_{\ell_1} |\ck|^{\ell_1} + |\cm|^{\ell_1}$ due to the convolution condition, we have
\begin{align*}
&\left|\ip{\pD^{\ell} \wt B_f (\umm_f,\vmm_s),\pD^{\ell} \umm}\right|\\
&  \lesssim_{\ell} |\T| \skm \sum_{\sigma_1 \sigma_3 <0} (|\ck|^{\ell_1}+ |\cm|^{\ell_1}) |\cm| |\cn|^{2\ell - {\ell_1}}   |v_m^0| |u_k^{\sigma_1}| |u_n^{\sigma_3}|  \ind_{\cN^{\textnormal {FSF}}}{ (k,m,n)}.
\end{align*}
Finally, we apply Lemma \ref{restrictedl} twice,  with $\mu = 3 - \frac {\xi} 2$,  in order to conclude, also recalling Remark \ref{restrictedr}. In particular, the indicator function $\ind_{\cN^{\textnormal {FSF}}}(\cdot,\cdot,\cdot)$ is symmetric with respect to permutations in its first and third arguments. 
\end{proof}
We now give a simple estimate for the SFF terms that we  will encounter in the proof of Theorem \ref{globalexistenceT}. The result is posed in terms of an unrestricted bilinearity, as the operator  $\wt B(\cdot, \cdot)$ includes  all SFF interactions.
\begin{lemma}\label{sff}
 Let $\beta_1\in(0,\frac 3 2)$.  Moreover, let $\umm \in H^{\frac 5 2 - \beta_1}(\T;\R^4)$ and $\vmm\in H^{\beta_1+1}(\T;\R^4)$ be divergence-free and zero-mean fields. Then the following estimate holds true:
\[
\left|\ip{B_f(\umm_s,\vmm_f), \Delta\vmm_f}\right|\lesssim_{\T,{ \beta_1}}  \|\umm_s\|_{H^{\frac 5 2 - \beta_1 }}\|\vmm_f\|_{H^1} \|\vmm_f\|_{H^{\beta_1+1}}.
\]
\end{lemma}
\begin{proof}
We use Parseval's identity, Remark \ref{urealR}  and the incompressibility of $\umm$  to get:
\begin{align*}
&-2\ip{B_f(\umm_s,\vmm_f), \Delta\vmm_f}\\
& =\ri|\T|  \skm \sum_{\sigma_2\sigma_3 \neq 0}{ u_k^0 u_m^{\sigma_2} v_n^{\sigma_3}}\cip{r_k^0}{\cm'} \cip{r_m^{\sigma_2}}{r_n^{\sigma_3}} (|\cn|^{2} - |\cm|^{2}).
\end{align*}
We have $\left| |\cn|-|\cm|\right| \leq |\ck|$ via the triangle inequality and  $|\cip{r_k^0}{\cm'}| \leq \min \{|\cm|,|\cn|\}$, due to incompressibility. Also, we clearly have $|\cn|+|\cm|{ \lesssim }\max\{|\cn|, |\cm|\}.$ Thus:
\begin{align*}
\left|\ip{B_f(\umm_s,\vmm_f), \Delta\vmm_f}\right|&  \leq |\T| \skm \sum_{\sigma_2\sigma_3 \neq 0}{ \left| u_k^0\right| \left| v_m^{\sigma_2}\right| \left|v_n^{\sigma_3}\right|} |\ck||\cm||\cn|.
\end{align*}
The result follows from \eqref{kps1}.
\end{proof}
\end{subsection}
\begin{subsection}{Proof of Theorem \ref{globalexistenceT}(Global energy bounds)}
Equipped with  Theorems \ref{restrictedfT}, \ref{restrictedT}, \ref{restrictedffT}, Lemma \ref{sff}, and the $L^2$ identity \eqref{L2}, we are now in a position to prove Theorem \ref{globalexistenceT}.  We will extensively use the standard interpolation inequality
\begin{equation}\label{interpolation}
\|u\|_{H^{\ell}} \lesssim  \|u\|_{H^{\ell_1}}^{\theta} \|u\|_{H^{\ell_2}}^{1 - \theta},\quad\text{with}\quad\ell=\theta_1\ell_1+(1-\theta)\ell_2,\quad  \theta_1,\theta_2\in (0,1),
\end{equation}
 throughout the following proof.  In addition, we will write $C_\delta$ and $C_{\delta^*}$ for the implied constants in the definition of $\delta$ and $\delta^*,$ respectively.
\begin{proof}[Proof of Theorem  \ref{globalexistenceT}]
\underline{\textbf{Estimates on the slow part}}
We first test \eqref{approxeqlpv} with $\wt Q$ and estimate the slow component in $L^2$, recalling \eqref{eqlpvslow} and Lemma \ref{orthogonalityL}. In particular, the mixed slow interactions  only consist of FFS triplets due to our choice of $\delta^*$. 
 Then, we use  \eqref{slowtransportgainL2}, \eqref{interpolation},  Corollary \ref{restrictedlpvC} with $\xi=\frac 6 5$,  $a= \frac 3 5$,  and Lemma \ref{ellipticity} to obtain:
\begin{align}\label{lpvest}
 \partial_t \| \wt Q\|_{L^2 }^2 + 2 \nu_{min} \|\wt Q\|_{H^1}^2 \leq  \partial_t \| \wt Q\|_{L^2 }^2 - 2  \ip{\wt{\nu}_{11}\Delta \wt Q, \wt Q}  & \leq 2\left| \ip{\lpv \wt B_s(\Ummt_f, \Ummt_f), \wt Q } \right|   \\
&\le 2C \|\wt \Umm_f \|_{H^1}^2 \|\wt Q\|_{L^2},\nonumber
\end{align}
 for a constant $C$ that only depends on ${ C_{\delta^*}},\eta, \T$. We immediately infer that
\[
 \partial_t \|\wt Q\|_{L^2}\leq {C}\|\wt \Umm_f \|_{H^1}^2.
\]
Integrating the latter  over $[0,T]$  and using \eqref{L2}, we obtain:
\begin{equation}\label{lpvest2}
 \|\wt Q(T)\|_{L^2}\leq {\frac {C\nu_{min}^{-1}} 2}  \|\wt \Umm_0\|_{L^2}^2 +\|\wt Q_0\|_{L^2},
\end{equation}
 where $\wt Q_0 = \lpv \wt \Umm_0$, so that  the $L^{\infty}_t H^1_x$ estimate in \eqref{mains} follows from \eqref{normequivalence}.
We insert the last estimate in \eqref{lpvest}  and use \eqref{L2} in order to deduce 
\begin{align}\label{lpvest1}
 2\nu_{min} \int_0^T \|\wt Q\|_{H^1}^2 \rd t \leq C\nu_{min}^{-1} \|\wt \Umm_0\|_{L^2}^2( C\nu_{min}^{-1} \|\wt \Umm_0\|_{L^2}^2 +\|\wt Q_0\|_{L^2})+\|\wt Q_0\|_{L^2}^2
\end{align}
and  the $L^{2}_t H^2_x$ estimate in \eqref{mains} follows from  \eqref{normequivalence}.
\\
\underline{\textbf{Estimates on the fast part}}
As far as the $H^1$ fast estimate is concerned,  we test \eqref{approxeqlpv1} with { $-\Delta\Ummt_f$} and use Lemma \ref{orthogonalityL} and Lemma \ref{ellipticity} to obtain:
\begin{align}\label{energyf}
 \partial_t \| \Ummt_f\|_{H^1 }^2 + 2 \nu_{min} \| \Ummt_f \|_{H^{2}}^2 & \leq  \partial_t \|  \Ummt_f \|_{H^{1} }^2 + 2 \ip{\wt{\nu}_{22}\Delta \Ummt_f, \Delta \Ummt_f}\\
& \leq 2\left|\ip{\wt  B_f (\Ummt, \Ummt),\Delta \Ummt_f }\right|.\nonumber
\end{align}
We split the fast term on the right according to the nature of the input modes for the bilinearity. First, using equation \eqref{restrictedfff1} for $\ell=1$, from Theorem \ref{restrictedfT}, and Young's inequality, we have
\begin{align}\label{ffsF}
 &\left|  \ip{\wt B_f (\Ummt_f, \Ummt_f),\Delta \Ummt_f }\right|  \lesssim_{\eta,\T,  {C_\delta}} \|\Ummt_f\|_{H^1}^2  \|\Ummt_f\|_{H^2}\leq {\frac C{\nu_{min}}}  \| \Ummt_f \|_{H^1}^4 + \frac{\nu_{min}} 6   \|\Ummt_f \|_{H^2}^2,
\end{align}
 for an absolute constant C.
We follow a similar procedure for the FSF terms, with the help of Theorem \ref{restrictedffT}  instead. In particular,  \eqref{restrictedfsf} with  $\xi=\frac 6 5$, $\ell=1$, $\ell_1= \frac 3 5$,  $a=a'= \frac 3 5$,  \eqref{normequivalence}, \eqref{interpolation},  \eqref{lpvest2} and Young's inequality   yield
 \begin{align}\label{ffsF1}
 \left|  \ip{\wt B_{ f} (\Ummt_f, \Ummt_s),\Delta \Ummt_f }\right|   & \lesssim_{\eta, \T, { C_{\delta^*}}} \left( \|\Ummt_s \|_{H^{\frac 8 5}}      \| \Ummt_f \|_{H^{\frac 3 5}}  +  \|\Ummt_s \|_{H^{1 }} \| \Ummt_f \|_{H^{\frac 6 5}} \right)     \| \Ummt_f \|_{H^2} \\
 &   \lesssim  \left({\|\wt Q \|_{H^1}}      \| \Ummt_f \|_{H^1}+ \| \wt Q\|_{L^2}  \| \Ummt_f \|_{H^1}^{\frac 4 5}  \| \Ummt_f \|_{H^2}^{\frac 1 5}\right)\| \Ummt_f \|_{H^2}  \nonumber\\
&\leq {\frac C{\nu_{min}}}(\|\wt Q\|_{H^1}^2 + \|\wt \Umm_f\|_{H^1}^2)\|\wt \Umm_f\|_{H^1}^2 +  \frac{\nu_{min}} 6 \|\wt \Umm_f\|_{H^2}^2,\nonumber
 \end{align}
  for an absolute constant $C$.
Finally, for the SFF terms  we  use Lemma \ref{sff} with $\beta_1=1$, \eqref{normequivalence}, \eqref{interpolation} and Young's inequality:
\begin{align}\label{ffsF2}
\left|  \ip{B_f (\Ummt_s, \Ummt_f),\Delta \Ummt_f }\right|    &  \lesssim_{\eta, \T} \|\Ummt_s \|_{H^{\frac 3 2}} \| \Ummt_f \|_{H^1}\| \Ummt_f \|_{H^2}\\&   \lesssim  {\| \Ummt_s\|_{H^2}} \| \Ummt_f \|_{H^1} \|\Ummt_f \|_{H^2} \nonumber\\
&\leq \frac{C}{\nu_{min}}\|\wt Q\|_{H^1}^2\|\wt \Umm_f\|_{H^1}^2+  \frac{\nu_{min}} 6 \|\wt \Umm_f\|_{H^2}^2,\nonumber
\end{align}
for an absolute constant C.
Combining \eqref{energyf} with \eqref{ffsF}, \eqref{ffsF1} and \eqref{ffsF2} yields
\begin{equation}\label{ffsF3}
\partial_t \|\wt \Umm_f\|_{H^1}^2+  {\nu_{min}}  \|\Ummt_f \|_{H^2}^2  \leq 
C(\eta, \T, { C_\delta, C_{\delta^*}}){\nu_{min}^{-1}}
\left(\|\wt Q \|_{H^1}^2 + \| \Ummt_f \|_{H^1}^2  \right)\|\wt\Umm_f\|_{H^1}^2.
\end{equation}
Then, the $L^\infty_t H^1_x$ estimate in \eqref{mainf2} follows by integrating the last inequality in time,   \eqref{L2}, \eqref{normequivalence},  \eqref{lpvest2} and   Gr\"onwall's inequality.  Finally, the $L^2_t H^2_x$ estimate in \eqref{mainf2} follows by combining  this estimate, \eqref{L2}, \eqref{normequivalence},  \eqref{lpvest2}  and \eqref{ffsF3}.
\end{proof}
\end{subsection}
\begin{subsection}{Proof of Theorem \ref{differenceT}(Error estimates)}
In this section, we prove Theorem  \ref{differenceT} on the difference of our approximation and the modulated system in an initial time interval.  The estimate that we obtain depends on $\nu_1, \nu_2$ only via  upper bounds for the ratio $\nu_R=\frac{\nu_{max}}{\nu_{min}}$ and $\nu_{max}$. Similar results have appeared concerning the proximity of the exact resonant dynamics to that of the full Boussinesq approximation in \cite{BMN2},\cite{BMN1}, \cite{BMNZ}  and  \cite{GAL1}. An interesting phenomenon is the higher regularity loss occurring in the difference equation, due to the presence of mixed  interactions. 

The strategy of the proof consists of conveniently expanding the difference of the approximate modulated and  modulated systems, along with standard tools, like Gr\"onwall's inequality. In order to apply the latter,  our control on $\omega_{kmn}^{\vec \sigma}$  outside the mixed near resonant set  proves crucial.  Then, the estimation concerning the fast terms proceeds in a similar manner to \cite{BCZNS}. On the other hand, a different lower bound for the bandwidth is implemented for the mixed terms. This is reflected on the larger derivative gap, $\ell-\ell'$, compared to the one of \cite[Theorem 1.4]{BCZNS}. Nevertheless, we only need standard bilinear estimates throughout the proof.
\begin{proof}[Proof of Theorem \ref{differenceT}]
 Under the  notation of Theorem \ref{differenceT}, 
let $\umm =e^{- \tau \Le}\Umm $  be a solution to  \eqref{boussmod0},  and let $\ummt:=e^{- \tau \Le}\Ummt$ be   a solution to the corresponding approximate system
 \begin{equation}\label{restrictedM}
\partial_t \ummt  + \wt B(\tau, \ummt,\ummt) +  \wt A \ummt = 0
\end{equation}
 which is the modulated version of our proposed approximate system \eqref{restrictedS}.  Note that $\wt A$ is acting on $\ummt$ in the same way as in \eqref{laplacianmod} due to the commutability properties of the restricted elliptic operators  in Section \ref{restrictedLB}.
We set $\wmm= \umm - \wt\umm$ and  $\Wmm=\Umm - \Ummt$. \\
\underline{\textbf{The equation for the difference}}\\
 First, we derive the equation  for $\wmm$:
\[
 \partial_ t \wmm  +B(\tau,\umm,\wmm)+ B(\tau,\wmm,\ummt)+ B(\tau,\ummt,\ummt)-\wt B(\tau,\ummt,\ummt)+  \wt A \wmm  + (A-\wt A)\umm=0. 
\]
The difference between the original  and modified bilinearities, occurring in the difference equation,  can be expressed as follows
\begin{align*}
&B(\tau,\ummt,\ummt)-\wt B(\tau,\ummt,\ummt)
\\ &= (B_s-\wt B_s)(\tau, \ummt_f,\ummt_f)   + (B_f - \wt B_f)(\tau, \ummt_f,\ummt_s) + (B_f- \wt B_f)(\tau, \ummt_f,\ummt_f)\\
&+ B_f (\tau,{ \ummt_s,\ummt_s}) +B_s(\tau, \ummt_f,\ummt_s)+B_s(\tau, \ummt_s,\ummt_f) . 
\end{align*}\\
\underline{\textbf{Elimination of oscillatory factors}}\\
We now focus on  terms in the difference equation containing time oscillations.  For the sake of brevity, we write $\ind_{\sigma_1\sigma_2 <0},\ind_{\sigma_1\sigma_2 >0}$ instead of introducing  separate summation signs under these restrictions. In particular,  we use the product rule in order to derive
\begin{align*}
& N  (B_s-\wt B_s)(\tau,\ummt_f, \ummt_f)= \partial_t \rmm_1 +  \rmmt_1\\ &:=   \partial_t \sum_{k,m ,n;conv} { \sum_{\sigma_1,\sigma_2}\left(\ind_{\sigma_1\sigma_2<0} \ind_{(\cN^{\textnormal {FFS}})^{\mathsf{c}}}+ \ind_{\sigma_1\sigma_2>0}\right)}(\ri  \omega_{kmn}^{\vec\sigma})^{-1}  e^{\ri\omega_{kmn}^{\vec \sigma}\tau}  B_{kmn}^{\sigma_1 \sigma_2 0}(\ummt_f,\ummt_f)  \\
& - \sum_{k,m ,n;conv}{ \sum_{\sigma_1,\sigma_2} \left(\ind_{\sigma_1\sigma_2<0} \ind_{(\cN^{\textnormal {FFS}})^{\mathsf{c}}}+ \ind_{\sigma_1\sigma_2>0}\right)}(\ri  \omega_{kmn}^{\vec\sigma})^{-1}\left( e^{\ri\omega_{kmn}^{\vec \sigma}\tau}   \partial_t B_{kmn}^{\sigma_1 \sigma_2 0}(\ummt_f,\ummt_f) \right).\\
\end{align*} 
A similar calculation for the FSF terms  yields 
\begin{align*}
& N  (B_f - \wt B_f)(\tau,\ummt_f, \ummt_s)= \partial_t \rmm_2 +  \rmmt_2\\
& :=\partial_t \sum_{k,m ,n;conv}{ \sum_{\sigma_1,\sigma_3} \left(\ind_{\sigma_1\sigma_3<0} \ind_{(\cN^{\textnormal {FSF}})^{\mathsf{c}}}+ \ind_{\sigma_1\sigma_3>0}\right)}(\ri  \omega_{kmn}^{\vec\sigma})^{-1} e^{\ri\omega_{kmn}^{\vec \sigma}\tau}  B_{kmn}^{\sigma_1 0 \sigma_3 }(\ummt_f,\ummt_s)  \\
& - \sum_{k,m ,n;conv}{\sum_{\sigma_1,\sigma_3}\left(\ind_{\sigma_1\sigma_3<0} \ind_{(\cN^{\textnormal {FSF}})^{\mathsf{c}}}+ \ind_{\sigma_1\sigma_3>0}\right)}  (\ri  \omega_{kmn}^{\vec\sigma})^{-1}\left( e^{\ri\omega_{kmn}^{\vec \sigma}\tau}   \partial_t B_{kmn}^{\sigma_1 0 \sigma_3}(\ummt_f,\ummt_s) \right).\\
\end{align*} 
Finally, the  FFF terms can be expressed as follows
\begin{align*}
& N  (B_f-\wt B_f)(\tau,\ummt_f, \ummt_f) = \partial_t \rmm_3 +  \rmmt_3\\
&: = \partial_t \sum_{k,m ,n;conv} \sum_{\sigma_1 \sigma_2 \sigma_3 \neq 0} \ind_{(\cN^{\textnormal {FFF}})^{\mathsf{c}}} (\ri  \omega_{kmn}^{\vec\sigma})^{-1} e^{\ri\omega_{kmn}^{\vec \sigma}\tau}  B_{kmn}^{\sigma_1 \sigma_2 \sigma_3}(\ummt_f,\ummt_f) \\
& - \sum_{k,m ,n;conv} \sum_{\sigma_1 \sigma_2 \sigma_3 \neq 0} \ind_{(\cN^{\textnormal {FFF}})^{\mathsf{c}}}  (\ri  \omega_{kmn}^{\vec\sigma})^{-1}\left( e^{\ri\omega_{kmn}^{\vec \sigma}\tau}   \partial_t B_{kmn}^{\sigma_1 \sigma_2 \sigma_3}(\ummt_f,\ummt_f) \right).
\end{align*} 
As far as the  SSF are concerned, we have
\begin{align*}
N B_f (\tau,\umm_s,\umm_s) =  \partial_t \rmm_{4} +  \rmmt_{4}& := \partial_t  \sum_{k,m ,n;conv} \sum_{\sigma_3 \neq 0}  (\ri  \omega_{kmn}^{\vec\sigma})^{-1} e^{\ri\omega_{kmn}^{\vec \sigma}\tau}    B_{kmn}^{00 \sigma_3}(\ummt_s,\ummt_s)\\
&-\sum_{k,m ,n;conv} \sum_{\sigma_3 \neq 0} (\ri  \omega_{kmn}^{\vec\sigma})^{-1} e^{\ri\omega_{kmn}^{\vec \sigma}\tau}   \partial_t B_{kmn}^{0 0 \sigma_3}(\ummt_s,\ummt_s).
\end{align*}
We also define in a similar manner
\[
B_s( \ummt_f,\ummt_s):=\partial_t \rmm_{5} +  \rmmt_{5}\quad\text{and}\quad B_s( \ummt_s,\ummt_f):=\partial_t \rmm_{6} +  \rmmt_{6}.
\]
Finally, we turn our attention to the oscillating viscosity and heat conductivity terms which appear in the difference equation, via  setting 
\begin{align*}
&N(A-\wt A)\umm= \partial_t \rmm_{7} +  \rmmt_{7}\\&:=-\partial_t  \Delta  \sum_{k\in\Z^3\setminus\{\vec 0\}} \sum_{\sigma\neq\sigma_1}\left(\ri \sigma_1 \omega_k -\ri \sigma \omega_k \right)^{-1}e^{\ri( \ck\cdot x +\sigma_1 \omega_k\tau -\sigma \omega_k \tau) }  \cip{\boldsymbol{\nu} r_k^{\sigma_1}}{\overline {r_k^{\sigma}}} u_k^{\sigma_1} r_k^{\sigma}\\
&{ +}\Delta \sum_{k\in\Z^3\setminus\{\vec 0\}} \sum_{\sigma\neq\sigma_1}  \left(\ri \sigma_1 \omega_k -\ri\sigma\omega_k  \right)^{-1} e^{\ri( \ck\cdot x +\sigma_1 \omega_k\tau - \sigma \omega_k \tau) } \cip{\boldsymbol{\nu} r_k^{\sigma_1}}{\overline {r_k^{\sigma}}} \partial_t u_k^{\sigma_1} r_k^{\sigma}.\nonumber
\end{align*}

In order to proceed, we set
\begin{equation}\label{wdef}
 \wmmt : = \wmm -N^{-1}\sum_{i=1}^7 {\rmm_i}.
\end{equation}
Then,  $ \wmmt$ satisfies the following equation
\begin{align}\label{difference1}
 \partial_ t \wmmt + A \wmmt &+  B(\tau,\umm,\wmmt) + B(\tau,\wmmt,\ummt) \\
&+  N^{-1} \sum_{i=1}^7 \left[ A \rmm_i +  B(\tau,\umm,\rmm_i) + B(\tau,\rmm_i,\ummt) + \rmmt_i\right]  =0.\nonumber
\end{align}

We remark that the presence of a negative power of the dispersion relation $\omega_{kmn}^{\vec \sigma}$ in the denominators of some of the previous expressions is not problematic, as the lower bound on the bandwidth provides us with sufficient control. In more detail, we have an estimate of the form 
\begin{equation}\label{far}
|\omega_{kmn}^{\vec \sigma}|^{-1}<  c_f^{-1}(|\ck|+|\cm|),\quad\text{when}\quad(k,m,n)\in (\cN^{\textnormal {FFF}})^{\mathsf{c}}.
\end{equation}
On the other hand, we have  a corresponding estimate
\begin{equation}\label{far1}
|\omega_{kmn}^{\vec \sigma}|^{-1}{ \leq c_s^{-1} C(\xi)}(|\ck|+|\cm|)^{\xi},\quad\text{when}\quad(k,m,n)\in (\cN^{\textnormal {FFS}})^{\mathsf{c}},
\end{equation}
with $\xi \in [\frac 6 5, 2]$   and $C(\xi)>0$. In addition, an identical estimate holds true in $(\cN^{\textnormal {FSF}})^{\mathsf{c}}$. \\
\underline{\textbf{Estimates on the time derivative}}\\
Next, we examine the regularity cost of estimating the time derivative of the solution to either the modulated or the approximate modulated systems.  In that direction, we use \eqref{kp}  for the bilinear term, so that
 \begin{align}\label{timederb}
 \| \partial_t \tu\|_{H^{\ell}} &\leq { \nu_{max}}\|\tu \|_{H^{\ell+2}} +\|\wt B(\tau,\tu,\tu) \|_{H^{\ell}}\\
& \leq \nu_{max} \|\tu \|_{H^{\ell+2}} + C(\ell')\| \tu\|_{H^{\frac 3 2 + \gamma}} \|\tu \|_{H^{\ell + 1}},\nonumber
\end{align}
for all $\gamma, \ell>0,$
with an identical estimate holding true for $\umm$ and $B(\tau,\cdot,\cdot)$.\\
\underline{\textbf{Estimates on the  remainder terms}}\\
We now examine the effect of \eqref{far} and \eqref{far1} on the terms $\rmm_i$, $\rmmt_i$, for { $i=1\ldots 6$}.  Since  $\xi>1$, we estimate $\rmm_1$ and $\rmmt_1$ only, with $\rmm_2,\rmmt_2$ obeying  similar estimates.  The terms $\rmm_3$ and $\rmmt_3$ can be treated analogously, up to a substitution of $\xi$ with 1.

  First,  we claim that
\[
\| \rmm_1 \|_{H^{\ell'}} \lesssim_{\ell',\xi,c_s^{-1}} \| \ummt \|_{\xi+\ell'+1} \| \ummt\|_{\frac 3 2+\gamma },\quad\text{for all}\quad\gamma>0.
\]
Indeed, we use \eqref{far1} and the convolution condition, in order to deduce that:
\[
\| \rmm_1 \|_{H^{\ell'}}^2 \lesssim_{\ell',\xi,c_s^{-1}} \skm |\wt u_k|^2 |\wt u_m|^2 |\cm|^{2}(|\ck| + |\cm|)^{2\xi} |\cn|^{2 \ell'},
\]
 with the claim following via \eqref{kp}.

Similarly,  using \eqref{timederb} and \eqref{kp} once more, together with \eqref{interpolation}, we deduce
\begin{align*}
  \| \rmmt_1\|_{H^{\ell'-1}} &\lesssim_{\ell',\xi,c_s^{-1}} \|\partial_t \ummt\|_{H^{\ell'+\xi+{\gamma}}}\| \ummt\|_{\frac 3 2  } + \|\partial_t  \ummt\|_{H^{\frac 3 2 }}\| \ummt\|_{H^{\ell'+ \xi+{\gamma}}}\\
&\leq C(\xi, \ell',c_s^{-1})  \left( \nu_{max} \| \ummt\|_{H^{\ell'+2+\xi +{\gamma}}}+\| \ummt\|_{H^{\ell'+1+\xi+\gamma}}\| \ummt\|_{H^{\frac 3 2 + \gamma}  }\right)\|\ummt \|_{H^{\frac 3 2}}\\
&+C(\xi, \ell',c_s^{-1})  \left(\nu_{max}\| \ummt\|_{H^{\frac 7 2}}  + \|\ummt\|_{H^{\frac 5 2}}\| \ummt\|_{H^{\frac 3 2 + \gamma}} \right) \|\ummt \|_{H^{\ell' + \xi +\gamma}}\\
&\leq C(\xi, \ell',c_s^{-1})  \left(\nu_{max}\| \ummt\|_{H^{\ell}}  + \|\ummt\|_{H^2}\| \ummt\|_{H^{\ell-1}} \right) \|\ummt \|_{H^{\frac 3 2}},
\end{align*}
for all $\gamma>0.$

 We now turn our attention to the remainder terms that do not come from a restriction on the level of interactions. Since   ${\left|\omega_{k}^\sigma\right|^{-1}} \leq \max\{1,\eta^{-1}\}$,  for $\sigma \in \{\pm\},$ we have bounds
\[
\|\rmm_4\|_{H^{\ell'}} \leq C(\eta,\ell',c_s^{-1}) \|\ummt \|_{H^{\frac 3 2 +\gamma}}\|\ummt \|_{H^{\ell'+1}},\quad\text{for all}\quad\gamma>0,
\]
using \eqref{kp} and the definition of $\rmm_4$ directly.
Moreover, $\rmm_5,\rmm_6$ can be bounded in a similar manner. The remaining terms,  $\rmmt_4, \rmmt_5, \rmmt_6$ can be estimated  without the bandwidth cost of order $\xi$.  In particular, we have:
\begin{align*}
  \| \rmmt_4\|_{H^{\ell'-1}} &\lesssim_{\ell',c_s^{-1}} \|\partial_t \ummt\|_{H^{\ell'+{\gamma}}}\| \ummt\|_{\frac 3 2  } + \|\partial_t  \ummt\|_{H^{\frac 3 2 }}\| \ummt\|_{H^{\ell'+{\gamma}}}\\
&\leq C(\ell',c_s^{-1})  \left(\nu_{max}\| \ummt\|_{H^{\ell-\xi}}  + \|\ummt\|_{H^2}\| \ummt\|_{H^{\ell-1-\xi}} \right) \|\ummt \|_{H^{\frac 3 2}},
\end{align*}
via  \eqref{interpolation}, \eqref{timederb}, \eqref{kp} and the definition of $\rmmt_4$ directly, with $\rmmt_5$ and $\rmmt_6$ satisfying similar estimates.\\
 \underline{\textbf{Estimates on the dissipative remainder  terms}}
As far as $\rmm_7$ and $\rmmt_7$ are concerned, we have no derivative losses, since $|\sigma_1 \omega_k -\sigma\omega_k|^{-1}\leq \frac 1 2 \max\{1,\eta^{-1}\}$ when $\sigma \neq \sigma_1.$ Thus, the following estimates hold true
\begin{equation}\label{r7}
\|\rmm_7\|_{H^\ell}\leq C \nu_{max}\|\umm\|_{\ell+2}\quad\text{and}\quad\|\rmmt_7\|_{H^\ell}\leq C \nu_{max}\|\partial_t\umm\|_{\ell+2},
\end{equation}
for all $\ell\in\R$ and a  constant $C>0$ depending on $\eta,\ell$.
\\
\underline{\textbf{Final arguments}}\\
We claim that the fields $\umm$ and $\ummt$ satisfy local in time estimates in  $L^\infty\left([0,T_0);H^{\ell}(\T;\R^4) \right) \cap L^2\left([0,T_0);H^{\ell+1}(\T;\R^4) \right)$, for a $T_0=T_0(E_0)$, as solutions to \eqref{boussmod0} and \eqref{restrictedM}, respectively. 
 In particular, there exists a constant $C>0$ and  a time $T_0>0$, depending on $E_0$, so that 
\begin{equation}\label{final0}
\sup_{t\in[0,T_0]}\left(\|\umm\|_{H^\ell}^2+\|\ummt\|_{H^\ell}^2\right) + 2\nu_{min}\int_0^{T_0}\left(\|\umm\|_{H^{\ell+1}}^2 + \|\ummt\|_{H^{\ell+1}}^2\right) \rd t \lesssim_{\ell} E_0,
\end{equation}
where $C=C(\ell,\ell_1)$.

Indeed, the original Boussinesq system has a Navier-Stokes type bilinearity and for the range of $\ell$ under consideration standard results apply, see e.g.  \cite{TAYLOR}. On the other hand, the approximation \eqref{restrictedM} can be handled similarly, due to the $L^2$ cancellation property of Corollary \ref{L2l0}.  In more detail, a $(\nu_1,\nu_2)$-independent local $H^\ell$ estimate for the approximate system requires the cancellation property $\ip{\wt B(\Umm,\pD^{\ell}\Vmm),\pD^{\ell}\Vmm}=0$ to handle the $(\ell+1)^{th}$ order derivatives, which is guaranteed by Corollary \ref{L2l0}. Next, we set 
\begin{itemize}
\item $\ell_\gamma: = \max \{\ell', \frac 5 2 + \gamma \}$
\item $\wt \ell_\gamma : = \max \{\ell' + 1, \frac 5 2 + \gamma \}$,
\end{itemize}
for $\gamma>0.$ The  bilinear interactions of \eqref{difference1} can be estimated similarly to the one in \eqref{timederb}. Testing the difference equation \eqref{difference1} with $\ds{\ell'} \wmm,$ we obtain
\begin{align*}
\partial_t \| \wmmt \|_{H^{\ell'}}^2 + 2 \nu_{min} \| \wmmt\|^2_{H^{\ell'+1}}&\leq C \left( \|\ummt\|_{H^{\wt \ell_\gamma}} +  \|\umm\|_{H^{ \ell_\gamma}}  \right) \| \wmmt\|_{H^{\ell'}}^2\\
& +C N^{-1}\sum_{i=1\ldots6}  \left( \|\umm\|_{H^{\wt \ell_\gamma -1}}  \|\rmm_i\|_{H^{\ell'+1}} + \|\ummt\|_{H^{\wt \ell_\gamma }}  \|\rmm_i\|_{H^{\ell'}}        \right)  \|\wmmt\|_{H^{\ell'}}\nonumber\\
& + CN^{-1} \left(   \sum_{i=1\ldots7} \left(   \|\rmmt_i\|_{H^{\ell' - 1}}+\|A \rmm_i\|_{H^{\ell' - 1}}\right)    \right)  \|\wmmt\|_{H^{\ell' +1 }}\nonumber\\
& + CN^{-1}     (\|\umm\|_{H^{\ell_\gamma -1}}\|\rmm_7\|_{H^{\ell'}}   + \|\ummt\|_{H^{\ell_\gamma  }}\|\rmm_7\|_{H^{\ell'-1}}  )      \|\wmmt\|_{H^{\ell' +1 }}\nonumber\\
\end{align*}
 for some $C=C(\eta, \ell',\xi,c_f^{-1},c_s^{-1})$, via  \eqref{kp} and H\"older's inequality. We examine the terms on the right separately. 
As far as the last term on the right is concerned, we use  Young's  inequality and \eqref{r7} to get:
\begin{align}\label{final1}
   &N^{-1}(\|\umm\|_{H^{\ell_\gamma -1}}\|\rmm_7\|_{H^{\ell'}}   + \|\ummt\|_{H^{\ell_\gamma  }}\|\rmm_7\|_{H^{\ell'-1}}  )      \|\wmmt\|_{H^{\ell' +1 }}\leq \frac {\nu_{min}} 3 \|\wmmt\|_{H^{\ell' +1 }}^2\\
&+ CN^{-2} \nu_R(\|\umm\|_{H^{\ell_\gamma -1}}^2\|\umm\|_{H^{\ell' +2}}^2   + \|\ummt\|_{H^{\ell_\gamma  }}^2\|\umm\|_{H^{\ell'+1}}^2  ).\nonumber
\end{align}
The remaining terms containing factors that depend on $\rmm_7$ and $\rmmt_7$ can be estimated using \eqref{timederb}, \eqref{r7} and Young's inequality:
\begin{align}\label{final2}
N^{-1}(  \|\rmmt_7\|_{H^{\ell' - 1}}&+\|A \rmm_7\|_{H^{\ell' - 1}})\|\wmmt\|_{H^{\ell' +1 }}\leq \nu_{max}\left( \|\partial_t \umm\|_{H^{\ell' +1}} + \|\umm\|_{H^{\ell' +3}}\right) \|\wmmt\|_{H^{\ell' +1 }}\\
&\lesssim_{\nu_{max},\ell'}N^{-2}\nu_R \left( \|\umm\|_{\ell'+3}^2+\|\umm\|_{\frac 3 2 + \gamma}^2\|\umm\|_{\ell'+2}^2\right) + \frac {\nu_{min}} 3\|\wmmt\|_{H^{\ell' +1 }}^2.\nonumber
\end{align}
Utilizing \eqref{timederb}, \eqref{r7} and Young's inequality once more, we obtain:
\begin{align*}
 &N^{-1} \left(    \left(   \|\rmmt_1\|_{H^{\ell' - 1}}+\|A \rmm_1\|_{H^{\ell' - 1}}\right)    \right)  \|\wmmt\|_{H^{\ell' +1 }}\leq  N^{-1} \left(  \left(   \|\rmmt_1\|_{H^{\ell' - 1}}+\nu_{max}\| \rmm_1\|_{H^{\ell' + 1}}\right)    \right)  \|\wmmt\|_{H^{\ell' +1 }}\\
 &\leq N^{-1} \left(\left(\nu_{max}\| \ummt\|_{H^{\ell}}  + \|\ummt\|_{H^2}\| \ummt\|_{H^{\ell-1}} \right) +\nu_{max}\| \ummt \|_{\xi+\ell'+2+\gamma} \right) \| \ummt\|_{\frac 3 2} \|\wmmt\|_{H^{\ell' +1 }}\nonumber\\
&\lesssim_{\nu_{max},\ell'} N^{-2}\nu_R \left(\| \ummt\|_{H^{\ell}}^2  + \|\ummt\|_{H^2}^2\| \ummt\|_{H^{\ell-1}} ^2 +\| \ummt \|_{\xi+\ell'+2+\gamma}^2 \right) \| \ummt\|_{\frac 3 2}^2 + \frac {\nu_{min}} {20}\|\wmmt\|_{H^{\ell' +1 }}^2. \nonumber
\end{align*}
It immediately follows that
\begin{align}\label{final3}
 N^{-1} &\left(   \sum_{i=1\ldots 6} \left(   \|\rmmt_i\|_{H^{\ell' - 1}}+\|A \rmm_i\|_{H^{\ell' - 1}}\right)    \right)  \|\wmmt\|_{H^{\ell' +1 }}\\
&\lesssim_{\nu_{max},\ell'} N^{-2}\nu_R  \left(\| \ummt\|_{H^{\ell}}^2  + \|\ummt\|_{H^2}^2\| \ummt\|_{H^{\ell-1}} ^2 +\| \ummt \|_{\xi+\ell'+2+\gamma}^2 \right) \| \ummt\|_{\frac 3 2}^2 + \frac {\nu_{min}} 3\|\wmmt\|_{H^{\ell' +1 }}^2.\nonumber
\end{align}

Inserting \eqref{final1}, \eqref{final2}, \eqref{final3} back into our energy estimates, estimating any remaining terms that contain $\rmm_i$ and using \eqref{final0} yields
\begin{align*}
&\partial_t \| \wmmt \|_{H^{\ell'}}^2 +  \nu_{min} \| \wmmt\|^2_{H^{\ell'+1}}\lesssim_{E_0}  \| \wmmt\|_{H^{\ell'}}^2 +N^{-2}.
\end{align*}
Then, in view of \eqref{wdef} and our bounds on the remainder terms $\rmm_i$, an integration in time and an appeal to  Gr\"onwall's inequality  complete the proof.
\end{proof}
\end{subsection}
\end{section}
\appendix
\section{Some product estimates}\label{appendixpe}
We  prove a weaker instance of the classical homogeneous Kato-Ponce fractional Leibniz rule,  see e.g. \cite{GK}, which is still sufficient for our purpose. In order to proceed, and for a given $f:\mathbb T^d\to \R$, we define:
\[
f_{abs}: = \sum_{n\in \Z^n} e^{\ri \cn\cdot x} |f_n|.
\]
\begin{lemma}\label{kpl}
Let $\ell\geq 0$ and $f,g:\mathbb T^d \to \R$ be sufficiently smooth functions with zero-mean. Then for all  $2\leq q,r,q',r' \leq \infty $ with $\frac 1 2 = \frac 1 q + \frac 1 r = \frac 1 {q'} + \frac  1 {r'}$  and all $a,b\geq 0$, the following estimate holds true
\begin{equation}\label{kp}
\| fg   \|_{H^{\ell}} \lesssim\| \pD^{\ell+a} f_{abs}  \|_{L^{q}}\| \pD^{-a}  g_{abs}  \|_{L^{r}}   + \| \pD^{-b} f_{abs}  \|_{L^{q'}} \| \pD^{\ell+b} g_{abs}  \|_{L^{r'}},
\end{equation}
with the implied constant only depending on $\ell.$
\end{lemma}
\begin{proof}We  use Plancherel's Theorem and a decomposition into high-low modes, in order to deduce that
\begin{align*}
\| fg   \|_{H^{\ell}}^2 &= \sum_{n\in\Z^3\setminus\{\vec 0\}} \left|\sum_{k\in\Z^3\setminus\{\vec 0\}} |\cn|^{\ell}  f_k g_m \right|^2\\
& \lesssim  \sum_{n\in\Z^3\setminus\{\vec 0\}} \left(\sum_{\substack{|\cm|\leq|\ck|}} |\ck|^{\ell+a} |\cm|^{-a} |f_k| |g_m| \right)^2 +  \sum_{n\in\Z^3\setminus\{\vec 0\}} \left(\sum_{\substack{|\ck|\leq|\cm|}} |\cm|^{\ell+b} |\ck|^{-b} |f_k| |g_m| \right)^2\\
&\leq C(\ell)\left[ \|\widehat{\pD^{\ell+a} f_{abs}} * \widehat {\pD^{-a}g_{abs}} \|_{\ell^2}^2 + \|\widehat{\pD^{-b} f_{abs}} * \widehat {\pD^{\ell+b}g_{abs}} \|_{\ell^2}^2 \right]\\
&= C(\ell)\left[\| \left( \pD^{\ell+a} f_{abs} \right)  \left( \pD^{-a}g_{abs} \right) \|_{L^2}^2 + \|  \left( \pD^{-b} f_{abs} \right)    \left(\pD^{\ell+b}g_{abs} \right) \|_{L^2}^2\right].
\end{align*}
Then the result follows by H\"older's inequality.
\end{proof}
\begin{corollary}\label{kpsc}
Let $f, g: \mathbb T^d \to \R$  with zero-mean, and $\beta_1, \beta_2  \in { [0,\frac d 2)}$ with $0<\beta_1+\beta_2$. Then the following estimate holds true:
\begin{equation}\label{kps}
\|f g \|_{H^{\beta_1+\beta_2-\frac d 2}} \lesssim \|f \|_{H^{\beta_1}} \| g \|_{H^{\beta_2}},
\end{equation}
with the implied constant  depending on $b_1, b_2, d.$
\end{corollary}
\begin{proof}
We use \eqref{kp} with $q=\frac{d}{\beta_2}$, $r=\frac{2d}{d-2\beta_2}$, $q'=\frac{d}{\beta_1}$, $r'=\frac{2d}{d-2\beta_1}$, and $a=b=0$. Then, we have
\begin{align*}
\|f g \|_{H^{\beta_1+\beta_2-\frac d 2}} &\lesssim  \|\pD^{\beta_1+\beta_2-\frac d 2} f \|_{L^{\frac{d}{\beta_2}}}\|g \|_{L^\frac{2d}{d-2\beta_2}} + \|f \|_{L^{\frac{d}{\beta_1}}}\|\pD^{\beta_1+\beta_2-\frac d 2} g \|_{L^{\frac{2d}{d-2\beta_1}}}.
\end{align*}
The result follows from Sobolev embedding, since the Sobolev norms of $f,g$ are equal to those of $f_{abs},g_{abs}$ respectively.
\end{proof}
Finally, we recall the following trilinear estimate.
\begin{lemma}
Let $f,g,h:\T \to \R$ be zero-mean functions of sufficient regularity and $\beta\in(0,\frac 3 2).$ Then the following estimate holds true
\begin{equation}\label{kps1}
\skm |f_k|\,|g_m|\,|h_n| \lesssim
  \|f\|_{H^ {\frac32-\beta}}  \|g\|_{H^\beta} \|h\|_{L^2}.
\end{equation}
\end{lemma}
\begin{proof}
The Cauchy-Schwarz inequality implies \[
\left(\skm |f_k|\,|g_m|\,|h_n| \right)^2\leq  \skm\left( |f_k|\,|g_m|\right)^2\|h\|_{L^2}^2
\]
and the result follows from Corollary \ref{kpsc}.
\end{proof}
\section{Lower Bounds}\label{appendixlb}
 The following simple lemma verifies the sharpness of the results of Corollary \ref{mixedvolumec}.  We assume $\sL_1=\sL_2=1$, as the general case is done similarly. 

\begin{lemma}
  Consider $\eta\ne1$ and $\delta^{*} \in \left(0,\min \{\frac \eta 2 , \frac 1 2\}\right)$.  Then, for any sufficiently large number $M$, the cardinality  of the set of wavevectors 
\[
\left\{k \in \Z^3:  M \le |k|<2M, |\omega_k - 1  | \leq \delta^{*}\right\}
\]
is bounded from below by $\frac{C\sqrt{\delta^{*}}}{\sqrt{|\eta^2 - 1|}} M^3 + C M^2$ for an absolute constant $C>0.$
\end{lemma}
Apparently, for any nonzero wavevector $n$ with $n_3=0$, any member of the above set satisfies $|\omega_k - \omega_n | \leq \delta^{*}$.

\begin{proof}
 
Let \begin{equation*} c_{\delta^{*}}= \begin{cases}
 \frac {\sqrt{-(\delta^{*})^2 + 2 \delta^{*}}} {\sqrt{1- \eta^2 }},\quad&\text{if}\quad\eta<1\\
 \frac {\sqrt{(\delta^{*})^2 + 2 \delta^{*}}} {\sqrt{\eta^2 - 1}},\quad&\text{if}\quad\eta>1.
 \end{cases}
 \end{equation*} Consider the family of $k$ defined via
\[
\left\{k
=(k_1,k_2,k_3)\in\Z^{3}: \frac35 M\le k_1\le M,\,\frac45 M \le k_2 \le M,\,0\le k_3\le c_{\delta^{*}}|k_H|\right\}.
\]
 Clearly,  $M\le |k|<2M$ and   $\frac {|k_3|}{|k|}<c_{\delta^{*}}$ the latter of which is equivalent to $|\omega_k - 1|\le\delta^{*}$.

For given $k_H \in \Z^2$, the possible integer choices for $k_3$ are $1 +\big\lfloor  c_{\delta^{*}}  |k_H| \big\rfloor> \frac 1 2\left( 1 + c_{\delta^{*}}|k_H|\right).$  We sum over $k_1, k_2 \in \left[\frac 35 M,M\right]\times  \left[\frac 45 M, M\right] $ in order to obtain a lower bound for the required cardinality
\[
\sum_{k_1,k_2}  \left(1+ c_{\delta^{*}}\right)  |k_H|  \gtrsim    \int_{\frac 45M}^M \int_{\frac 35M}^M   {\left(1+c_{\delta^{*}}\right)}  |k_H|  \rd k_1 \rd k_2 > \frac{C\sqrt{\delta^{*}}}{\sqrt{|\eta^2 - 1|}}M^3 + CM^2.
\]
\end{proof}

\section{Interaction coefficients}\label{appendixs}
 We provide the following  on the form of FFS interaction coefficients.
\begin{lemma}\label{scof}
Let $ (k,m,n) \in \Z^9$.   The interaction coefficients  $S_{kmn}^{+-0}$ defined in \eqref{slowcoeff} for FFS interactions  are given as follows. 
\begin{itemize}
\item[(i)] If $\ck_H\ne\vec0$ and $\cm_H\ne\vec0$ then
\begin{align*}
 S_{kmn}^{+-0}&= \ri \left[  \omega_m^2 - \omega_k^2  \right]\dfrac{ \left[   |\ck|^2 |\cm|^2 (1-\eta^2)^{-1}   (\ck_H \times \cm_H)  (\eta^2 - \omega_k\omega_m)    \right]}{2|\ck_\eta| |\cm_\eta| |\cn_\eta|  |\ck_H| |\cm_H|}\\
&+  \left[ (\omega_m - \omega_k)  \right]\dfrac{\eta  \left[ (k_3^2 |\cm_H|^2  + m_3^2 |\ck_H|^2) ( \cm_H \cdot \ck_H) -2  k_3 m_3 |\ck_H|^2 |\cm_H|^2  \right]}{2|\ck_\eta| |\cm_\eta| |\cn_\eta|  |\ck_H| |\cm_H|}.
\end{align*}
\item[(ii)]  If $\ck_H=\vec0$ and $\cm_H\ne\vec0$ then
\[
 S_{kmn}^{+-0}=   \ri \left[  \omega_m^2 - \omega_k^2  \right]\dfrac{   k_3 |\cm|^2 (1-\eta^2)^{-1} (\eta - \omega_m)   (\ri \cm_1 + \cm_2) }{2 |\cm_\eta| |\cn_\eta|   |\cm_H|}.
\]
  The case for $\ck_H\ne\vec0$ and $\cm_H=\vec0$ is similar due to the $k,m$ symmetry of $\ind_{\cN_{FFS}}(k,m,n).$
\item[(iii)]
 If $\ck_H=\cm_H=\vec0$ then
$S_{kmn}^{+-0}=0$.
\end{itemize}
\end{lemma}
\begin{proof} We will use the eigenvectors given in \eqref{defeigenvs1} and \eqref{defeigenvs}.

\noindent \underline{Case (i).}

We proceed in estimating the two summands in the interaction coefficients separately.
First, using  Lemma \ref{properties} and the incompressibility of $r^0_m$, we have
\begin{align*}
&2|\ck_\eta| |\cm_\eta| |\cn_\eta|  |\ck_H| |\cm_H| \cip{\rfp}{\cm'} \cip{\rfmm }{r_n^0 } = - \cip{\alpha_k}{\cm' } \cip{\overline \alpha_m }{e_k^0 } \\
& =  \left[  k_3 \left( \eta  \cm_H \times \ck_H  - \ri  \omega_k + (\ck_H \cdot \cm_H)  \right)  + \ri  \omega_k |\ck_H|^2 m_3 \right]  \left[  m_3 \left(   \ri  \omega_m(\cm_H \times \ck_H) + \eta(\cm_H \cdot \ck_H) \right)  -  \eta k_3 |\cm_H|^2   \right] \\
& =  A_1 + A_2 +  A_3 + A_4,
\end{align*} 
where
\begin{itemize}
\item $A_1 =k_3  m_3  (\ck_H \cdot \cm_H )  \left[ \eta^2\cm_H \times \ck_H +  \omega_k\omega_m   \right]$
\item $ A_2 = \ri \eta \omega_k (\ck_H \cdot \cm_H)  \left[k_3^2 |\cm_H|^2 +\eta m_3 ^2 |\ck_H|^2  \right]$
\item $A_3 = -  (\cm_H \times \ck_H) \left[  \eta^2 k_3^2 |\cm_H|^2  + \omega_k \omega_m m_3^2 |\ck_H|^2  \right]  $  
\item $ A_4 = \ri   \eta  k_3 m_3 \left[ \omega_m |\ck_H \times \cm_H|^2 - \omega_k ( |\ck_H \cdot \cm_H|^2 + |\ck_H|^2 |\cm_H|^2)  \right].$
\end{itemize}

As far as the second part of the sum is concerned, we have:
\begin{align*}
& 2|\ck_\eta| |\cm_\eta| |\cn_\eta|  |\ck_H| |\cm_H|\cip{\rfmm}{\ck' } \cip{\rfp}{r_n^0 }  = - \cip{\overline\alpha_m}{\ck' } \cip{\alpha_k}{e_m^0 } \\
& = \left[  m_3 \left( \eta \ck_H \times \cm_H  + \ri  \omega_m (\ck_H \cdot \cm_H)  \right)  - \ri  \omega_m |\cm_H|^2 k_3 \right]  \left[  k_3 \left(   \ri  \omega_k (\cm_H \times \ck_H) + \eta(\cm_H \cdot \ck_H) \right)  - \eta m_3 |\ck_H|^2   \right]  \\
& =   B_1+B_2 + B_3 +B_4,
\end{align*} 
with
\begin{itemize}
\item $B_1 =k_3  m_3  (\ck_H \cdot \cm_H )  \left[ \eta^2\ck_H \times \cm_H -  \omega_k \omega_m   \right]$
\item $ B_2 =- \ri \eta \omega_m (\ck_H \cdot \cm_H)  \left[\eta k_3^2 |\cm_H|^2 + m_3 ^2 |\ck_H|^2  \right]$
\item $B_3 =  (\cm_H \times \ck_H) \left[  \eta^2 m_3^2 |\ck_H|^2  + \omega_k\omega_m k_3^2 |\cm_H|^2  \right]  $  
\item $ B_4 = \ri   \eta  k_3 m_3 \left[  \omega_k |\ck_H \times \cm_H|^2 +\omega_m ( |\ck_H \cdot \cm_H|^2 + |\ck_H|^2 |\cm_H|^2)  \right].$
\end{itemize}
The symmetry between $k,m$ then implies that:
\[
A_3 + B_3 =(\eta^2 -\omega_k\omega_m) (\ck_H \times \cm_H)  \left[  k_3^2 |\cm_H|^2     -  m_3^2 |\ck_H|^2  \right].
\]
Following \cite{AOW}, we observe that the definition of $\omega$ allows us to rewrite
\[
\omega_k^2 - \omega_m^2 =(1-\eta^2) \dfrac {|\ck_H|^2 |\cm|^2 - |\cm_H|^2 |\ck|^2}{|\ck|^2 |\cm|^2}    =(1-\eta^2) \dfrac {|\ck_H|^2 m_3^2 - |\cm_H|^2 k_3^2}{|\ck|^2 |\cm|^2}.
\]
Thus,  we obtain:
\[
 A_3 + B_3  = \frac {|\ck|^2 |\cm|^2} {(1-\eta^2)}   (\ck_H \times \cm_H)  (\eta^2 - \omega_k\omega_m) \left[  \omega_m^2 - \omega_k^2  \right].
\]
Finally,  the cancellation between $A_1, B_1$, together with the sums of the remaining terms 
\begin{equation*}
A_2 + B_2 = - \ri \eta (\ck_H \cdot \cm_H)(m_3 |\ck_H|^2 + k_3^2 |\cm_H|^2)(\omega_m - \omega_k)
\end{equation*}
and
\begin{equation*}
A_4 + B_4 = 2 \ri \eta k_3 m_3  |\ck_H|^2|\cm_H|^2 (\omega_m - \omega_k),
\end{equation*}
 yield the result.
 
 \noindent \underline{Case (ii).}

We have
\begin{align*}
2 |\cm_\eta| |\cn_\eta|   |\cm_H| \cip{r^+_k}{m'} \cip{\rfmm }{r_n^0 }&=  \cip{ \alpha_k}{m' } \cip{\overline \alpha_m }{e_n^0 } =  \eta (\ri \cm_1 + \cm_2)k_3 |\cm_H|^2
\end{align*}
and
\begin{align*}
2 |\cm_\eta| |\cn_\eta|   |\cm_H| \cip{\rfmm}{k'} \cip{ r^+_k }{r_n^0 }&=  -\cip{\overline\alpha_m}{k' } \cip{  \alpha_k }{e_m^0 } = (\ri \cm_1 + \cm_2)k_3 |\cm_H|^2 \omega_m,
\end{align*}
 after using Lemma \ref{properties}.
The result then follows, since  the wavevectors under consideration satisfy
\[
\omega_m^2 - \omega_k^2 = \dfrac {(1-\eta^2)|\cm_H|^2}{|\cm|^2}.
\]

\noindent \underline{Case (iii).} Trivial.

\end{proof}

{ \section*{\bf Acknowledgements}

Cheng and Sakellaris are supported by the Leverhulme Trust (Award No.\,RPG-2017-098). Cheng is supported by the EPSRC (Grant No.\,EP/R029628/1).  The authors also thank Beth Wingate for insightful discussions and valuable feedback.
}
\bibliography{Manuscript}
\bibliographystyle{plain}
\end{document}